%% file: ms.tex
\documentclass[english,final,pdftex,reqno]{article}
 \pdfoutput=1
\usepackage{amsmath,amsthm,amssymb}
\usepackage{mathtools}
\usepackage{tikz}
\usepackage{bbm}
\usepackage{enumitem}
\usepackage[total={420pt,680pt},centering,includeheadfoot]{geometry}
\usepackage[plainpages=false,hypertexnames=false,colorlinks=true]{hyperref}

\usepackage[capitalize,nameinlink]{cleveref}

\crefformat{section}{#2\S#1#3}
\crefformat{subsection}{#2\S#1#3}

\Crefname{figure}{Figure}{Figures}

\crefformat{equation}{\textup{#2(#1)#3}}
\crefrangeformat{equation}{\textup{#3(#1)#4--#5(#2)#6}}
\crefmultiformat{equation}{\textup{#2(#1)#3}}{ and \textup{#2(#1)#3}}
{, \textup{#2(#1)#3}}{, and \textup{#2(#1)#3}}
\crefrangemultiformat{equation}{\textup{#3(#1)#4--#5(#2)#6}}%
{ and \textup{#3(#1)#4--#5(#2)#6}}{, \textup{#3(#1)#4--#5(#2)#6}}{, and \textup{#3(#1)#4--#5(#2)#6}}

\Crefformat{equation}{#2Equation~\textup{(#1)}#3}
\Crefrangeformat{equation}{Equations~\textup{#3(#1)#4--#5(#2)#6}}
\Crefmultiformat{equation}{Equations~\textup{#2(#1)#3}}{ and \textup{#2(#1)#3}}
{, \textup{#2(#1)#3}}{, and \textup{#2(#1)#3}}
\Crefrangemultiformat{equation}{Equations~\textup{#3(#1)#4--#5(#2)#6}}%
{ and \textup{#3(#1)#4--#5(#2)#6}}{, \textup{#3(#1)#4--#5(#2)#6}}{, and \textup{#3(#1)#4--#5(#2)#6}}

\crefdefaultlabelformat{#2\textup{#1}#3}

\input{macros}

\input{envs_article}

\title{A class of abstract delay differential equations\\ in the light of suns and stars}

\author{Sebastiaan G. Janssens\thanks{Department of Mathematics, Utrecht University, Budapestlaan 6, 3508 TA Utrecht, The Netherlands\newline e-mail: \href{mailto:s.g.janssens@uu.nl}{s.g.janssens@uu.nl} or \href{mailto:sj@dydx.nl}{sj@dydx.nl}}}

\renewcommand{\phi}{\varphi}

\begin{document}

\maketitle

\begin{abstract}
  Using dual perturbation theory in a non-sun-reflexive context, we establish a correspondence between 1. a class of nonlinear abstract delay differential equations (DDEs) with unbounded linear part and an unknown taking values in an arbitrary Banach space and 2. a class of abstract {\WSTAR} integral equations of convolution type involving the sun-star adjoint of a translation-like strongly continuous semigroup. For this purpose we also characterize the sun dual of the underlying state space. More generally we consider bounded linear perturbations of an arbitrary strongly continuous semigroup and we comment on some implications for the particular case of abstract DDEs.
\end{abstract}

\input{introduction}
\input{duality}
\input{sunstar}
\input{conclusion}

\newpage
\appendix
\numberwithin{equation}{section}
\numberwithin{theorem}{section}

\input{integration}

\bibliographystyle{hplain}
\bibliography{bibliography}

\end{document}

%% file: macros.tex
\newcommand{\NN}{\mathbb{N}}

\newcommand{\RR}{\mathbb{R}}
\newcommand{\CC}{\mathbb{C}}

\newcommand{\KK}{\mathbb{K}}

\DeclareMathOperator{\RE}{\textup{Re}}

\DeclareMathOperator{\DOM}{\mathcal{D}}

\newcommand{\EPS}{\ensuremath{\varepsilon}}
\newcommand{\BAR}[1]{\ensuremath{\overline{#1}}}

\newcommand{\PAIR}[3][]{\ensuremath{\langle #2,#3 \rangle_{#1}}}

\newcommand{\WSTAR}{\ensuremath{\text{weak}^{\ast}}}
\newcommand{\WSTARLY}{\ensuremath{\text{weakly}^{\ast}}}

\newcommand{\WLIM}{\textup{w$^{\ast}$-}\lim}

\newcommand{\LP}[1]{\ensuremath{L^{#1}}}

\newcommand{\BND}{\ensuremath{\mathcal{L}}}

\newcommand{\BV}{\ensuremath{\mbox{BV}}}
\newcommand{\NBV}{\ensuremath{\text{NBV}}}

\newcommand{\C}[1]{\ensuremath{\mathcal{C}_{#1}}}
\newcommand{\STAR}[1]{\ensuremath{#1^{\ast}}}
\newcommand{\SUN}[1]{\ensuremath{#1^{\odot}}}
\newcommand{\SUNSTAR}[1]{\ensuremath{#1^{\odot\star}}}
\newcommand{\SUNSUN}[1]{\ensuremath{#1^{\odot\odot}}}
\newcommand{\STARSTAR}[1]{\ensuremath{#1^{\ast\ast}}}

\newcommand{\DEF}{\ensuremath{\coloneqq}}

\newcommand{\QEDEX}{\ensuremath{\lozenge}}

%% file: envs_article.tex
\theoremstyle{plain}
\newtheorem{theorem}{Theorem}[]
\newtheorem*{theorem*}{Theorem}
\newtheorem{corollary}[theorem]{Corollary}
\newtheorem{lemma}[theorem]{Lemma}
\newtheorem{proposition}[theorem]{Proposition}
\theoremstyle{remark}
\newtheorem{remark}[theorem]{Remark}
\theoremstyle{definition}
\newtheorem{definition}[theorem]{Definition}

\newenvironment{steps}
{\begin{enumerate}[wide,labelindent=0pt,itemsep=0.5em,label={{\it \arabic*.}}]}
{\end{enumerate}}

%% file: introduction.tex
\section{Introduction}\label{sec:adde:intro}
We are concerned with the initial value problem
\begin{subequations}
  \label{eq:adde:1}
  \begin{alignat}{2}
    \dot{x}(t) &= B x(t) + F(x_t), &\qquad &t \ge 0,\label{eq:adde:dde}\\
    x(\theta) &= \phi(\theta), &&\theta \in [-h,0],\label{eq:adde:ic}
  \end{alignat}
\end{subequations}
for the \emph{abstract} delay differential equation (DDE) \cref{eq:adde:dde} with initial condition \cref{eq:adde:ic} and $0 < h < \infty$. The unknown $x$ takes values in a real or complex Banach space $Y$ and $B : \DOM(B) \subseteq Y \to Y$ is the generator of a \C{0}-semigroup $S$ of bounded linear operators on $Y$. It is \emph{not} assumed that $S$ is analytic, nor (immediately or eventually) compact. As state space we choose the Banach space $X \DEF C([-h,0];Y)$  of continuous $Y$-valued functions on the interval $[-h,0]$, endowed with the supremum norm. The {\bf history} $x_t \in X$ at time $t \ge 0$ is defined as
\[
  x_t(\theta) \DEF x(t + \theta), \quad \forall\,\theta \in [-h,0],
\]
so $x_t$ is the translation to $[-h,0]$ of the restriction of $x$ to $[t - h,t]$. Finally, the nonlinear operator $F : X \to Y$ is continuous. 

If $Y$ is finite-dimensional then \cref{eq:adde:dde} reduces to what will be called a \emph{classical} DDE. For this case a rather complete dynamical theory based on perturbation theory for adjoint semigroups \cite{Clement1987, Clement1988, Clement1989, Clement1989b, Diekmann1991} is available in \cite{Diekmann1995}. (The theory is also known in the literature as \emph{dual perturbation theory} or \emph{sun-star calculus}.) It was recently understood that various more general and seemingly different classes of delay equations can be formulated and analyzed within the same functional analytic framework, largely independently of the details of a particular class. It is only in the choice of the underlying state space that these details matter. Therefore, dual perturbation theory can play the useful role of a unifying device.

For example, in \cite{Diekmann2007} purely functional equations - also called renewal equations (REs) - as well as coupled RE-DDE systems are investigated for the case that the delay and the dimension of $Y$ are both finite. In \cite{Diekmann2012} the analysis is extended to include infinite delay, while in \cite{Diekmann2008} abstract REs with infinite delay and a possibly infinite-dimensional Banach space $Y$ are considered. A lot of this work was motivated by the author's interest in models of structured biological populations, but their results have a general validity. In this article we will show that \eqref{eq:adde:1}, too, fits naturally into the framework of dual perturbation theory.

On the other hand, there already exists a substantial literature on \cref{eq:adde:1} for the case that $Y$ is infinite-dimensional, driven in part - but not exclusively - by examples where $B$ generates a diffusion semigroup. We mention \cite{Travis1974}, \cite{Wu1996}, \cite{Faria2002} and \cite{Faria2006}, and also \cite{Arino2006} for the bounded linear case. These works build upon the formal duality approach towards classical DDEs  \cite{Hale1971, Hale1977, Hale1993}. For systematic development of the linear theory in state spaces of continuous functions or integrable functions, respectively, we also mention \cite[\S VI.6]{Engel2000} and \cite{Batkai2005}.  Furthermore, there are classes of delay equations, such as neutral DDEs, for which recent approaches inspired by - but well beyond the realm of - dual perturbation theory seem promising \cite{Diekmann2019}.

The original motivation of the present article can be found in \cite{VanGils2013} where it appears as a reference. Indeed, part of the work discussed here had been completed around the time of publication of \cite{VanGils2013}, but unfortunately it was not finished until now. Some results were previously presented in Twente in June 2014 as well as in Leiden in April 2013 during the workshop \emph{Mathematics and Biology: a Roundtrip in the Light of Suns and Stars}.

\section{Structure and outline}
The primary purpose of this article is to show that \cref{eq:adde:1} fits naturally in the framework of dual perturbation theory, also when $B$ is unbounded. We establish a relationship between \cref{eq:adde:1} and an abstract integral equation of the form
  \begin{equation}
    \label{eq:adde:2}
    u(t) = T_0(t)\phi + j^{-1}\int_0^t{\SUNSTAR{T}_0(t - \tau)G(u(\tau))\,d\tau}, \qquad \forall\,t \ge 0,
  \end{equation}
  where $T_0$ is a \C{0}-semigroup on $X$, $\STAR{T_0}$ is the adjoint semigroup on the dual space $\STAR{X}$ and $\SUN{T_0}$ is the restriction of $\STAR{T_0}$ to the invariant maximal subspace $\SUN{X}$ of strong continuity of $\STAR{T_0}$. The operator $G : X \to \SUNSTAR{X}$ is continuous and the convolution integral is of {\WSTAR} Riemann type with values in $\SUNSUN{X}$. The canonical embedding $j : X \to \SUNSTAR{X}$ defined by
  \begin{equation}
    \label{eq:adde:4}
    \PAIR{\SUN{x}}{jx} \DEF \PAIR{x}{\SUN{x}}, \qquad \forall\,x \in X,\,\SUN{x} \in \SUN{X},
  \end{equation}
  actually takes values in $\SUNSUN{X}$. If $j$ maps \emph{onto} $\SUNSUN{X}$ then $X$ is said to be {\bf sun-reflexive} with respect to $T_0$. In that case the application of $j^{-1}$ in \cref{eq:adde:2} is automatically warranted. Part of the task of relating \cref{eq:adde:1,eq:adde:2} is in showing that, although $X$ will in general \emph{not} be sun-reflexive with respect to $T_0$, for abstract DDEs the corresponding convolution integral in \cref{eq:adde:2} nevertheless takes values in the range of $j$.

As a starting point of the perturbative approach to \cref{eq:adde:1} we consider the trivial initial value problem
\begin{equation}
  \label{eq:duality:1}
  \left\{
    \begin{aligned}
      \dot{x}(t) &= Bx(t), &&t \ge 0,\\
      x(\theta) &= \phi(\theta), &&\theta \in [-h,0],
    \end{aligned}
  \right.
\end{equation}
which is \cref{eq:adde:1} with $F = 0$. Clearly the solution $x : [-h, \infty) \to Y$ of \cref{eq:duality:1} is obtained from the unique mild solution \cite[Definition II.6.3]{Engel2000} of the abstract Cauchy problem associated with $B$ and $\phi(0)$, so
\[
  x_0 = \phi, \qquad x(t) = S(t)\phi(0), \qquad \forall\,t \ge 0.
\]
This solution defined the strongly continuous \textbf{shift semigroup} $T_0$ on $X$ as
\begin{equation}
  \label{eq:adde:3}
  (T_0(t)\phi)(\theta) \DEF x(t + \theta) = 
  \begin{cases}
    \phi(t+\theta), &-h \le t + \theta \le 0,\\
    S(t + \theta)\phi(0), &\hphantom{-} 0 \le t + \theta,
  \end{cases}
  \quad \forall\,\phi \in X,\, t \ge 0,\, \theta \in [-h,0].
\end{equation}

In \cref{sec:duality:sun_dual} we obtain a representation for the sun dual $\SUN{X}$ of $X$ with respect to the shift semigroup $T_0$. This representation is included in \cref{fig:duality:1}. In the process we also find representations for the duality pairing between $X$ and $\SUN{X}$ and for the semigroup $\SUN{T_0}$ acting on $\SUN{X}$. For various ideas of proofs, this section is indebted to the earlier work \cite{Greiner1992}.

In \cref{sec:duality:range} we are concerned with the convolution integral appearing in \cref{eq:adde:2}. In certain contexts, the natural choice for $Y$ may be a Banach space that is non-reflexive, see for instance \cite{VanGils2013}. In such cases $X$ is not sun-reflexive with respect to the shift semigroup $T_0$ and one has to verify explicitly that the convolution integral takes values in the range of $j$.

In \cref{sec:adde:adde_as_aie} we then combine the results from the foregoing sections to prove that, for a suitably chosen operator $G$, there is a one-to-one correspondence between solutions of \cref{eq:adde:1} and \cref{eq:adde:2}. As a consequence of this correspondence, the existence and uniqueness of solutions of \cref{eq:adde:1} is a rather straightforward matter.

It then becomes very natural and tempting to conjecture (or even: claim \cite{Diekmann2008}) that basic results from dynamical systems theory such as the principle of linearized stability, theorems on local invariant manifolds \cite{Diekmann1991,Diekmann1995} and theorems on local bifurcations \cite{VanGils1984,Diekmann1995,Janssens2010} hold as well for the particular class \cref{eq:adde:1} of abstract DDEs. Still, strictly speaking these theorems were formulated under the running assumption of sun-reflexivity. It is therefore of some importance to check at a general level if and how they change - and if weaker conditions may need to be imposed - when this assumption no longer holds. In \cref{sec:sunstar:linear} this work is started by considering the simple case of bounded linear perturbations of a general \C{0}-semigroup, with the shift semigroup as a motivating example.

In \cref{sec:sunstar:linadde} we return to abstract DDEs in order to make the connection with the general results from the previous section. We also comment on the ranges of the linear and nonlinear perturbations of the shift semigroup $T_0$ corresponding to linear and nonlinear abstract DDEs. Unlike in the classical case, these perturbations are not of finite rank, but they still turn out to be rather simple. In particular, their ranges are contained in a known closed subspace of $\SUNSTAR{X}$.

In \cref{sec:integration} we collect and prove some facts about vector-valued functions of bounded variation and bilinear Riemann-Stieltjes integration. Probably all of them are well-known, but it is convenient to have them available in one place. They are used repeatedly in the main text, mostly in \cref{sec:duality:sun_dual,sec:duality:range,sec:adde:adde_as_aie}.

We would like to mention that an introduction to dual perturbation theory that is sufficient for this article can be found in \cite[Appendix II.3]{Diekmann1995}, while \cite{VanNeerven1992} offers a more profound treatment. Throughout $\KK$ stands for the field of real or complex numbers and we abbreviate $\RR_+ \DEF [0,\infty)$ and $\RR_- \DEF (-\infty,0]$. Functionals act from right to left, so if $W$ is a Banach space with dual space $\STAR{W}$ then
\[
\PAIR{w}{\STAR{w}} \DEF \STAR{w}(w), \qquad \forall\,w \in W,\, \STAR{w} \in \STAR{W}.
\]
If $V$ and $W$ are Banach spaces, then $\BND(V,W)$ is the Banach space of all bounded linear operators from $V$ to $W$, equipped with the operator norm.


%% file: duality.tex
\section{Characterization of the sun dual}\label{sec:duality:sun_dual}
The dual space $\STAR{X}$ admits a characterization in terms of familiar functions. (This is in contrast with the situation for purely functional equations \cite[\S 2]{Diekmann2008}.) Namely, by \cref{def:integration:3,thm:integration:2} we can represent $\STAR{X}$ by the Banach space $\NBV([0,h],\STAR{Y})$ consisting of $\STAR{Y}$-valued functions of bounded variation on $[0,h]$.
\begin{figure}[ht]
  \centering
  \begin{tikzpicture}[baseline=(current bounding box.center)]
    \draw [thick, ->] (0.2,2.5) -- (4.8,2.5);
    \node [left] at (0,2.5) {$C([-h,0],Y) = X$};
    \node [right] at (5,2.5) {$\STAR{X} \simeq \NBV([0,h],\STAR{Y})$};
    %
    \draw [thick, <-] (0.2,-2.5) -- (4.8,-2.5);
    \node [right] at (5,-2.5) {$\SUN{X} \simeq \SUN{Y} \times \LP{1}([0,h],\STAR{Y})$};
    \node [left] at (0,-2.5) {$\SUNSTAR{Y} \times \STAR{[\LP{1}([0,h],\STAR{Y})]} \simeq \SUNSTAR{X}$};
    %
    \draw [thick, ->] (5.3,2.0) -- (5.3,-2.0);
    %
    \draw [thick, ->] (-0.3,-2.0) -- (-0.3,-0.5);
    \node [right] at (-0.6,0) {$\SUNSUN{X}$};
    \draw [thick, <-] (-0.3,0.5) -- (-0.3,2.0);
    \node [left] at (-0.3,1.25) {$j$};
  \end{tikzpicture}    
  \caption{The sun-star duality structure. The symbol $\simeq$ indicates an isometric isomorphism.}
  \label{fig:duality:1}
\end{figure}
Let $\LP{1}([0,h],\STAR{Y})$ be the Banach space of all Bochner integrable $\STAR{Y}$-valued functions on $[0,h]$. On this space we define the nilpotent \C{0}-semigroup $T_1$ by translation to the left with extension by zero,
\begin{displaymath}
  (T_1(t)g)(\theta) =
  \begin{cases}
    g(t + \theta), & 0 \le t + \theta \le h,\\
    0,             & h < t + \theta,
  \end{cases}
  \qquad \forall\,g \in \LP{1}([0,h],\STAR{Y}),\, t \ge 0,\, \theta \in [0,h].
\end{displaymath}
We recall from \cref{def:integration:4} that $\chi_0$ denotes the characteristic function on $(0,h]$. For real numbers $p$ and $q$ we let $p \wedge q \DEF \min\{p,q\}$ and $p \vee q \DEF \max\{p,q\}$. We are now ready to formulate the main result of this subsection.
\begin{theorem}
  \label{thm:duality:1}
  The maximal subspace of strong continuity of the adjoint semigroup $\STAR{T_0}$ is
   \begin{equation}
     \label{eq:duality:9}
    \begin{aligned}
    \SUN{X} = \Bigl\{f : [0,h] \to \STAR{Y} \,:\, & \text{ there exist } \SUN{y} \in \SUN{Y} \text{ and } g \in \LP{1}([0,h],\STAR{Y})\\
&\text{ such that } f(t) = \chi_0(t)\SUN{y} + \int_0^t{g(s)\,ds}\,\,\forall\,t \in [0,h]\Bigr\},
  \end{aligned}
  \end{equation}
  and $\iota : \SUN{Y} \times \LP{1}([0,h],\STAR{Y}) \to \SUN{X}$ defined by
  \begin{equation}
    \label{eq:duality:16}
     \iota(\SUN{y},g)(t) \DEF \chi_0(t)\SUN{y} + \int_0^t{g(s)\,ds}, \qquad \forall\,t \in [0,h],
  \end{equation}
  is an isometric isomorphism. The duality pairing between $\phi \in X$ and $\SUN{\phi} \DEF \iota(\SUN{y},g) \in \SUN{X}$ is given by
  \begin{equation}
    \label{eq:duality:3}
    \PAIR{\phi}{\SUN{\phi}} = \PAIR{\phi(0)}{\SUN{y}} + \int_0^h{\PAIR{\phi(-\theta)}{g(\theta)}\,d\theta}.
  \end{equation}
  For the action of $\SUN{T_0}$ on $\SUN{\phi} \DEF \iota(\SUN{y},g) \in \SUN{X}$ we have
  \begin{equation}
    \label{eq:duality:7}
    \SUN{T_0}(t)\SUN{\phi} = \iota(\SUN{S}(t)\SUN{y} + \int_0^{t \wedge h}{\STAR{S}(t - \theta)g(\theta)\,d\theta}, T_1(t)g),
  \end{equation}
  where the integral in the right-hand side is a weak$^\ast$ Lebesgue integral with values in $\SUN{Y}$.
\end{theorem}
Before proving the above theorem we first state and prove two lemmas.
\begin{lemma}
  \label{lem:duality:4}
  Let $E$ be defined as the right-hand side of \cref{eq:duality:9}. For every $\phi \in X$, $f \in E$ and $t \ge 0$,
  \begin{equation}
    \label{eq:duality:15}
    \begin{aligned}
      \PAIR{\phi}{\STAR{T_0}(t)f} = \PAIR{S(t)\phi(0)}{\SUN{y}} &+ \int_0^{t \wedge h}{\PAIR{S(t - \theta)\phi(0)}{g(\theta)}\,d\theta}\\
      &+ \int_0^h{\PAIR{\phi(-\theta)}{(T_1(t)g)(\theta)}\,d\theta},
    \end{aligned}
  \end{equation}
  where the Lebesgue integrals are taken to be zero when the upper and lower limits coincide.
\end{lemma}
\begin{proof}
  For $t = 0$ the statement follows from \cref{prop:integration:4}, so we may suppose that $t > 0$. Using \cref{eq:adde:3} we obtain
  \begin{align*}
    \PAIR{\phi}{\STAR{T_0}(t)f} = \PAIR{T_0(t)\phi}{f} &= \int_0^h{(T_0(t)\phi)(-\theta)\,df(\theta)}\\
    &= \int_0^{t \wedge h}{S(t - \theta)\phi(0)\,df(\theta)} + \int_{t \wedge h}^h{\phi(t - \theta)\,df(\theta)}.
  \end{align*}
  Now \cref{prop:integration:4} implies that
  \begin{displaymath}
    \int_0^{t \wedge h}{S(t - \theta)\phi(0)\,df(\theta)} = \PAIR{S(t)\phi(0)}{\SUN{y}} + \int_0^{t \wedge h}{\PAIR{S(t - \theta)\phi(0)}{g(\theta)}\,d\theta}.
  \end{displaymath}
  By the same proposition we have, at first for $t \in (0,h)$,
  \begin{align*}
    \int_{t \wedge h}^h{\phi(t - \theta)\,df(\theta)} &= \int_{t \wedge h}^h{\PAIR{\phi(t - \theta)}{g(\theta)}\,d\theta}\\
    &= \int_0^{h - t}{\PAIR{\phi(-\theta)}{g(t + \theta)}\,d\theta}\\
    &= \int_0^h{\PAIR{\phi(-\theta)}{(T_1(t)g)(\theta)}\,d\theta},
  \end{align*}
  but in fact 
  \begin{displaymath}
    \int_{t \wedge h}^h{\phi(t - \theta)\,df(\theta)} = \int_0^h{\PAIR{\phi(-\theta)}{(T_1(t)g)(\theta)}\,d\theta}, \qquad \forall\,t > 0,
  \end{displaymath}
  since for $t \ge h$ the left-hand side vanishes while $T_1(t)g = 0$.
\end{proof}

The next lemma gives a representation for the resolvent of the adjoint generator $\STAR{A_0}$ of $\STAR{T_0}$.

\begin{lemma}
  \label{lem:duality:6}
  For $\lambda \in \CC$ with $\RE{\lambda}$ sufficiently large, let $R(\lambda,\STAR{A_0})$ and $R(\lambda,\STAR{B})$ be the resolvents of $\STAR{A_0}$ and $\STAR{B}$ at $\lambda$. For $f \in \STAR{X}$ define
  \begin{displaymath}
    \SUN{y} \DEF R(\lambda,\STAR{B})\int_0^h{e^{-\lambda \theta}\,df(\theta)} \in \DOM(\STAR{B}) \subseteq \SUN{Y},
  \end{displaymath}
  and 
  \begin{displaymath}
    g(s) \DEF e^{\lambda s}\int_s^h{e^{-\lambda\theta}\,df(\theta)}, \qquad \forall\,s \in [0,h].
  \end{displaymath}
  Then $g \in \LP{1}([0,h],\STAR{Y})$ and 
  \begin{equation}
    \label{eq:duality:26}
    [R(\lambda,\STAR{A_0})f](t) = \chi_0(t)\SUN{y} + \int_0^t{g(s)\,ds}, \qquad \forall\,t \in [0,h].
  \end{equation}
\end{lemma}
\begin{proof}
  Let $R(\lambda,A_0)$ and $R(\lambda,B)$ be the resolvents of $A_0$ and $B$ at $\lambda$. These resolvent are guaranteed to exist for $\RE{\lambda}$ sufficiently large, because $A_0$ and $B$ are generators of \C{0}-semigroups. As a special case of \cite[Proposition VI.6.7]{Engel2000} we have
  \begin{displaymath}
    [R(\lambda,A_0)\phi](\theta) = e^{\lambda \theta}R(\lambda,B)\phi(0) + \int_{\theta}^0{e^{\lambda(\theta - s)}\phi(s)\,ds}, \qquad \forall\,\phi \in X,\, \theta \in [-h,0].
  \end{displaymath}
  Since $R(\lambda,\STAR{A_0}) = \STAR{R(\lambda,A_0)}$ we have for every $\phi \in X$ and every $f \in \STAR{X}$ that
  \begin{align}
    \PAIR{\phi}{R(\lambda,\STAR{A_0})f} &= \PAIR{R(\lambda,A_0)\phi}{f}\nonumber\\
    &= \int_0^h{e^{-\lambda \theta}R(\lambda,B)\phi(0)\,df(\theta)} + \int_0^h{\int_{-\theta}^0{e^{-\lambda(\theta + s)}\phi(s)\,ds}\,df(\theta)}. \label{eq:duality:13}
  \end{align}
  By \cref{prop:integration:5} the first term in \cref{eq:duality:13} equals
  \begin{displaymath}
    \PAIR{R(\lambda,B)\phi(0)}{\int_0^h{e^{-\lambda \theta}\,df(\theta)}} = \PAIR{\phi(0)}{R(\lambda,\STAR{B})\int_0^h{e^{-\lambda \theta}\,df(\theta)}} = \PAIR{\phi(0)}{\SUN{y}}.
  \end{displaymath}
  Meanwhile, by Fubini's theorem the second term in \cref{eq:duality:13} equals
  \begin{align*}
    \int_0^h{\int_s^h{e^{-\lambda(\theta - s)}\phi(-s)\,df(\theta)}\,ds} &= \int_0^h{\PAIR{\phi(-s)}{e^{\lambda s}\int_s^h{e^{-\lambda\theta}\,df(\theta)}}\,ds} && \text{by \cref{prop:integration:5}}\\
    &= \int_0^h{\PAIR{\phi(-s)}{g(s)}\,ds}.
  \end{align*}
  For any $s \in [0,h]$ we have by \cref{prop:integration:1,prop:integration:2} that
  \begin{displaymath}
    g(s) = \tilde{g}(s) - f(s), \qquad \forall\,s \in [0,h],
  \end{displaymath}
  for a function $\tilde{g} : [0,h] \to \STAR{Y}$ that is continuous, hence Bochner integrable. \Cref{prop:integration:6} implies that $f$ is Bochner integrable as well, so $g \in \LP{1}([0,h],\STAR{Y})$. From \cref{eq:duality:13} we conclude that
  \begin{displaymath}
    \PAIR{\phi}{R(\lambda,\STAR{A_0})f} = \PAIR{\phi(0)}{\SUN{y}} + \int_0^h{\PAIR{\phi(-s)}{g(s)}\,ds}, \qquad \forall\,\phi \in X.
  \end{displaymath}
  \Cref{prop:integration:4} then yields \cref{eq:duality:26}.
\end{proof}
\begin{proof}[Proof of \cref{thm:duality:1}]
  For clarity we proceed in a number of steps. We will first prove \cref{eq:duality:9} by showing two inclusions. Next we verify the statement regarding the isometry, which will also show that \cref{eq:duality:3} is true. We conclude by proving \cref{eq:duality:7}, including the statement about the range of the \WSTAR-integral appearing there. 
  \begin{steps}
  \item
    Let $f \in E$ be arbitrary, where we recall from \cref{lem:duality:4} that $E$ is defined as the right-hand side of \cref{eq:duality:9}. We prove that $f \in \SUN{X}$. As we are concerned with the limit $t \downarrow 0$, in this step we may assume without loss of generality that $t \in (0,h)$ whence $t \wedge h = t$. Invoking \cref{prop:integration:4} once more we also have
    \begin{equation}
      \label{eq:duality:5}
      \PAIR{\phi}{f} = \PAIR{\phi(0)}{\SUN{y}} + \int_0^h{\PAIR{\phi(-\theta)}{g(\theta)}\,d\theta},
    \end{equation}
    so from \cref{eq:duality:5,eq:duality:15} we obtain
    \begin{subequations}
      \begin{align}
        |\PAIR{\phi}{\STAR{T_0}(t)f} - \PAIR{\phi}{f}| &\le |\PAIR{S(t)\phi(0)}{\SUN{y}}  - \PAIR{\phi(0)}{\SUN{y}}|\label{eq:duality:10}\\
                                                       &+ \Bigl|\int_0^h{\PAIR{\phi(-\theta)}{(T_1(t)g)(\theta)}\,d\theta} - \int_0^h{\PAIR{\phi(-\theta)}{g(\theta)}\,d\theta}\Bigr|\label{eq:duality:11}\\
                                                       &+ \Bigl|\int_0^t{\PAIR{S(t - \theta)\phi(0)}{g(\theta)}\,d\theta}\Bigr|.\label{eq:duality:12}
      \end{align}
    \end{subequations}
    If $\|\phi\| = 1$ then \cref{eq:duality:10} equals
    \begin{displaymath}
      |\PAIR{\phi(0)}{\STAR{S}(t)\SUN{y}}  - \PAIR{\phi(0)}{\SUN{y}}| \le \|\STAR{S}(t)\SUN{y} - \SUN{y}\| \to 0 \qquad \text{as } t \downarrow 0,
    \end{displaymath}
    while \cref{eq:duality:11} does not exceed
    \begin{displaymath}
      \int_0^h{\|(T_1(t)g)(\theta) - g(\theta)\|\,d\theta} \to 0 \qquad \text{ as } t \downarrow 0,
    \end{displaymath}
    by the strong continuity of $T_1$. Concerning \cref{eq:duality:12} we observe that
    \begin{displaymath}
      |\PAIR{S(t - \theta)\phi(0)}{g(\theta)}| \le \|S(t - \theta)\|\cdot\|g(\theta)\| \le M e^{|\omega|t}\|g(\theta)\| \qquad \text{a.e. } \theta \in [0,t],
    \end{displaymath}
    for certain $M \ge 1$ and $\omega \in \RR$ depending only on $S$. Since $M e^{|\omega|t}$ stays bounded as $t \downarrow 0$ it follows from the integrability of $\|g\|$ on $[0,h]$ and the scalar dominated convergence theorem that \cref{eq:duality:12} tends to zero as $t \downarrow 0$, uniformly for $\|\phi\| = 1$. We conclude that $f \in \SUN{X}$.
  \item
    We next prove that $\SUN{X} \subseteq E$. Following an idea appearing in the proof of \cite[Theorem 2.2]{Greiner1992}, we invoke \cref{eq:duality:26} in \cref{lem:duality:6} to conclude that $R(\lambda,\STAR{A_0})\STAR{X} \subseteq E$, which is sufficient in light of the fact that $\SUN{X} = \BAR{\DOM(\STAR{A_0})}$ and $E$ is closed. (The closedness of $E$ will be proved independently in the next step.) 
  \item
    \Cref{lem:duality:1} shows that $\iota$ defined by \cref{eq:duality:16} is an isometry onto $E$. We note that this implies in particular that $E$ is closed in $\STAR{X}$. The equality in \cref{eq:duality:3} is simply \cref{eq:duality:15} with $f = \SUN{\phi}$ and $t = 0$.
  \item
    It remains to prove \cref{eq:duality:7}. Let $\phi \in X$ be arbitrary. By \cref{eq:duality:15} with $f = \SUN{\phi}$ we have
    \begin{equation}
      \label{eq:duality:4}
      \begin{aligned}
        \PAIR{\phi}{\SUN{T_0}(t)\SUN{\phi}} = \PAIR{\phi(0)}{\SUN{S}(t)\SUN{y} &+ \int_0^{t \wedge h}{\STAR{S}(t - \theta){g(\theta)}\,d\theta}}\\
        &+ \int_0^h{\PAIR{\phi(-\theta)}{(T_1(t)g)(\theta)}\,d\theta},
      \end{aligned}
    \end{equation}
    where we note the appearance of the {\WSTAR}-Lebesgue integral. Since $\SUN{X}$ is $\SUN{T_0}$-invariant, there exist $\SUN{v} \in \SUN{Y}$ and $w \in \LP{1}([0,h],\STAR{Y})$ such that $\SUN{T_0}(t)\SUN{\phi} = \iota(\SUN{v},w)$. By \cref{eq:duality:3,eq:duality:4} it then follows that
    \begin{align*}
      \PAIR{\phi(0)}{\SUN{S}(t)\SUN{y} + \int_0^{t \wedge h}{\STAR{S}(t - \theta){g(\theta)}\,d\theta} - \SUN{v}} + \int_0^h{\PAIR{\phi(-\theta)}{(T_1(t)g)(\theta) - w(\theta)}\,d\theta} = 0,
    \end{align*}
    for all $\phi \in X$. Now let $y \in Y$ be arbitrary and let $(\phi_n)_n$ be a uniformly bounded sequence in $X$ such that $\phi_n(0) = y$ for all $n \in \NN$ and $\phi_n$ converges a.e. to zero. Substituting $\phi_n$ for $\phi$ in the above identity, taking the limit $n \to \infty$ and using the scalar dominated convergence theorem on the rightmost integral, one obtains
    \begin{displaymath}
      \PAIR{y}{\SUN{v}} = \PAIR{y}{\SUN{S}(t)\SUN{y}} + \int_0^{t \wedge h}{\PAIR{y}{\STAR{S}(t - \theta){g(\theta)}}\,d\theta},
    \end{displaymath}
    and since $y$ was arbitrary, we infer that
    \begin{displaymath}
      \SUN{v} = \SUN{S}(t)\SUN{y} + \int_0^{t \wedge h}{\STAR{S}(t - \theta){g(\theta)}\,d\theta},
    \end{displaymath}
    so in particular the {\WSTAR}-integral takes values in $\SUN{Y}$. It now follows from \cref{eq:duality:3,eq:duality:4} that
    \begin{displaymath}
      \PAIR{\phi}{\SUN{T_0}(t)\SUN{\phi}} = \PAIR{\phi}{\iota(\SUN{S}(t)\SUN{y} + \int_0^{t \wedge h}{\STAR{S}(t - \theta){g(\theta)}\,d\theta},T_1(t)g)},
    \end{displaymath}
    for all $\phi \in X$, which is precisely \cref{eq:duality:7}. \qedhere
  \end{steps}
\end{proof}

\begin{remark}
  \label{rem:duality:1}
  The elements $\SUN{y}$ and $g$ in the right-hand side of \cref{eq:duality:9} are unique. To see this, suppose that the equality
  \begin{displaymath}
    f(t) = \chi_0(t)\SUN{y} + \int_0^t{g(\theta)\,d\theta}, \qquad \forall\,t \in [0,h],
  \end{displaymath}
  also holds with $\SUN{y}$ replaced by $\SUN{y_1}$ and $g$ by $g_1$. By subtraction we obtain
  \begin{displaymath}
    \| \SUN{y} - \SUN{y_1}\| \le \int_0^t{\|g(\theta) - g_1(\theta)\|\,d\theta}
  \end{displaymath}
  for $t > 0$. By taking the limit $t \downarrow 0$ and using the scalar dominated convergence theorem we find that $\SUN{y} = \SUN{y_1}$. Consequently it holds that
  \begin{displaymath}
    \int_0^t{[g(\theta) - g_1(\theta)]\,d\theta} = 0, \qquad \forall\,t \in [0,h],
  \end{displaymath}
  so $g = g_1$ a.e. on $[0,h]$ by \cite[Proposition 1.2.2a]{Arendt2011}. \hfill \QEDEX
\end{remark}

\begin{remark}
  \label{rem:duality:2}
  From now on we will identify $\SUN{X}$ with $\SUN{Y} \times \LP{1}([0,h],\STAR{Y})$. This is justified by the existence of the isometric isomorphism $\iota$ of \cref{thm:duality:1}. So, we write $\SUN{\phi} = (\SUN{y},g) \in \SUN{X}$, express the duality pairing of $\phi \in X$ with $\SUN{\phi} \in \SUN{X}$ as in \cref{eq:duality:3} and suppress $\iota$ in \cref{eq:duality:7}. \hfill \QEDEX
\end{remark}

\section{The range of the convolution integral}\label{sec:duality:range}
We now turn our attention to the convolution integral appearing in \cref{eq:adde:2}. Throughout, $T_0$ will be the shift semigroup defined by \cref{eq:adde:3}. (For the reader familiar with \cite{Diekmann2008} it may be interesting to see the close parallels between the forthcoming and the line of thought leading to \cite[Corollary 2.4]{Diekmann2008}, although the technical details are different.) Define $\delta \in \BND(\SUN{X},\SUN{Y})$ by 
\begin{equation}
  \label{eq:duality:28}
  \delta \SUN{\phi} \DEF \SUN{y}, \qquad \forall\,\SUN{\phi} = (\SUN{y},g) \in \SUN{X}.
\end{equation}
Let $j_Y$ be the canonical embedding of $Y$ into $\SUNSTAR{Y}$,
\begin{equation}
  \label{eq:duality:27}
  \PAIR{\SUN{y}}{j_Yy} \DEF \PAIR{y}{\SUN{y}}, \qquad \forall\,y \in Y,\, \SUN{y} \in \SUN{Y},
\end{equation}
cf. \cref{eq:adde:4}, and define
\begin{equation}
  \label{eq:duality:19}
  \ell \DEF \STAR{\delta} j_Y \in \BND(Y,\SUNSTAR{X}),
\end{equation}
It follows directly from the foregoing definitions that
\begin{equation}
  \label{eq:duality:2}
  \PAIR{y}{\delta \SUN{\phi}} = \PAIR{\SUN{\phi}}{\ell y}, \qquad \forall\,\SUN{\phi} \in \SUN{X},\, y \in Y.
\end{equation}

\begin{lemma}
  \label{lem:duality:8}
  $\ell y = (j_Yy, 0)$ for all $y \in Y$. Consequently $\ell$ is an isomorphism onto its closed range.
\end{lemma}
\begin{proof}
  For any $\SUN{\phi} = (\SUN{y},g) \in \SUN{X}$,
  \begin{displaymath}
    \PAIR{\SUN{\phi}}{\ell y} = \PAIR{y}{\delta \SUN{\phi}} = \PAIR{y}{\SUN{y}} = \PAIR{\SUN{y}}{j_Yy} = \PAIR{\SUN{\phi}}{(j_Yy,0)}
  \end{displaymath}
  where \cref{eq:duality:2} was used in the first equality. The second statement is clear, with the closedness of $\ell(Y)$ in $\SUNSTAR{X}$ following from the fact that $j_Y$ has closed range in $\SUNSTAR{Y}$.
\end{proof}

We now turn to the discussion of the range of the convolution integral in \cref{eq:adde:2}.

\begin{lemma}
  \label{lem:duality:2}
  For $\SUN{\phi} = (\SUN{y},g) \in \SUN{X}$ and $y \in Y$ we have
  \begin{displaymath}
    \PAIR{\SUN{\phi}}{\SUNSTAR{T_0}(t)\ell y} = \PAIR{y}{\SUN{S}(t)\SUN{y}} + \int_0^{t \wedge h}{\PAIR{S(t - \theta)y}{g(\theta)}\,d\theta}.
  \end{displaymath}
\end{lemma}
\begin{proof}
  It holds that
  \begin{align*}
    \PAIR{\SUN{\phi}}{\SUNSTAR{T_0}(t)\ell y} = \PAIR{\SUN{T_0}(t)\SUN{\phi}}{\ell y} &= \PAIR{y}{\delta\SUN{T_0}(t)\SUN{\phi}} && \text{by \cref{eq:duality:2}}\\
    &= \PAIR{y}{\SUN{S}(t)\SUN{y}} + \int_0^{t \wedge h}{\PAIR{S(t - \theta)y}{g(\theta)}\,d\theta}, && \text{by \cref{eq:duality:7}}
  \end{align*}
  as claimed.
\end{proof}

The next result shows that for a certain class of continuous functions on $\RR_+$ with values in the range of $\ell$, the {\WSTAR} convolution integral with $\SUNSTAR{T_0}$ takes values in the range of $j$. The integral also satisfies a standard estimate.

\begin{proposition}
  \label{prop:duality:1}
  Let $f : \RR_+ \to Y$ be continuous, let $t \ge 0$ and define $\psi \in X$ by
  \begin{displaymath}
    \psi(\theta) \DEF \int_0^{(t + \theta)^+}{S(t - \tau + \theta)f(\tau)\,d\tau}, \qquad \forall\,\theta \in [-h,0],
  \end{displaymath}
  where $(t + \theta)^+ \DEF (t + \theta) \vee 0$. Then
  \begin{equation}
    \label{eq:duality:8}
    \int_0^t{\SUNSTAR{T_0}(t - \tau)\ell f(\tau)\,d\tau} = j\psi,
  \end{equation}
  where $j$ is the canonical embedding of $X$ into $\SUNSTAR{X}$ given by \cref{eq:adde:4}. Moreover, there exist constants $M \ge 1$ and $\omega \in \RR$ such that
  \begin{equation}
    \label{eq:duality:18}
    \Bigl\|j^{-1}\int_0^t{\SUNSTAR{T_0}(t - \tau)\ell f(\tau)\,d\tau}\Bigr\| \le M \frac{e^{\omega t} - 1}{\omega} \sup_{0 \le \tau \le t}{\|f(\tau)\|},
  \end{equation}
  where the factor $(e^{\omega t} - 1)\omega^{-1}$ must be interpreted as its limiting value $t$ in case $\omega = 0$.
\end{proposition}
\begin{proof}
  We first prove \cref{eq:duality:8} and then show the estimate \cref{eq:duality:18}.
  \begin{steps}
  \item
  For any $\SUN{\phi} = (\SUN{y},g) \in \SUN{X}$ we have
  \begin{subequations}
    \begin{align}
      \PAIR{\SUN{\phi}}{\int_0^t{\SUNSTAR{T_0}(t - \tau)\ell f(\tau)\,d\tau}} &= \int_0^t{\PAIR{\SUN{\phi}}{\SUNSTAR{T_0}(t - \tau)\ell f(\tau)}\,d\tau}\nonumber\\
      &= \int_0^t{\PAIR{f(\tau)}{\SUN{S}(t - \tau)\SUN{y}}\,d\tau}\label{eq:duality:14}\\
      &+ \int_0^t{\int_0^{(t - \tau) \wedge h}{\PAIR{S(t - \tau - \theta)f(\tau)}{g(\theta)}\,d\theta}\,d\tau}, \label{eq:duality:17}
    \end{align}
  \end{subequations}
  where \cref{lem:duality:2} was used for the second equality. We see that \cref{eq:duality:14} is equal to
  \begin{displaymath}
     \int_0^t{\PAIR{S(t-\tau)f(\tau)}{\SUN{y}}\,d\tau} = \PAIR{\int_0^t{S(t-\tau)f(\tau)\,d\tau}}{\SUN{y}} = \PAIR{\psi(0)}{\SUN{y}}.
  \end{displaymath}
  Let $\tilde{g} \in \LP{1}(\RR_+,\STAR{Y})$ be the extension of $g$ to $\RR_+$ by zero, i.e. $\tilde{g}(\theta) \DEF 0$ for a.e. $\theta > h$. Then by Fubini's theorem \cref{eq:duality:17} equals
  \begin{align*}
   \int_0^t{\int_0^{t - \tau}{\PAIR{S(t - \tau - \theta)f(\tau)}{\tilde{g}(\theta)}\,d\theta}\,d\tau} &= \int_0^t{\int_0^{t - \theta}{\PAIR{S(t - \tau -\theta)f(\tau)}{\tilde{g}(\theta)}\,d\tau}\,d\theta}\\
    &= \int_0^t{\PAIR{\int_0^{t-\theta}{S(t - \tau - \theta)f(\tau)\,d\tau}}{\tilde{g}(\theta)}\,d\theta}\\
    &= \int_0^{t \wedge h}{\PAIR{\int_0^{t-\theta}{S(t - \tau - \theta)f(\tau)\,d\tau}}{g(\theta)}\,d\theta}\\
    &= \int_0^h{\PAIR{\int_0^{(t-\theta)^+}{S(t - \tau - \theta)f(\tau)\,d\tau}}{g(\theta)}\,d\theta}\\
    &= \int_0^h{\PAIR{\psi(-\theta)}{g(\theta)}\,d\theta},
  \end{align*}
  so we conclude that
  \begin{displaymath}
    \PAIR{\SUN{\phi}}{\int_0^t{\SUNSTAR{T_0}(t - \tau)\ell f(\tau)\,d\tau}} = \PAIR{\psi}{\SUN{\phi}}.
  \end{displaymath}
  Since $\SUN{\phi}$ was arbitrary, this yields \cref{eq:duality:8}.
\item
  Let $M \ge 1$ and $\omega \in \RR$ be such that $\|S(t)\| \le M e^{\omega t}$ for all $t \ge 0$. Then for any $\theta \in [-h,0]$ we have
  \begin{align*}
    \|\psi(\theta)\| &\le \int_0^{(t + \theta)^+}{\|S(t - \tau + \theta)f(\tau)\|\,d\tau} \le M \sup_{0 \le \tau \le t}{\|f(\tau)\|} \int_0^t{e^{(t - \tau + \theta)\omega}\,d\tau}\\
    &= M\frac{e^{\omega\theta}(e^{\omega t} - 1)}{\omega} \sup_{0 \le \tau \le t}{\|f(\tau)\|} \le M \frac{e^{\omega t} - 1}{\omega} \sup_{0 \le \tau \le t}{\|f(\tau)\|},
  \end{align*}
  where in the last estimate a factor $e^{|\omega|h}$ was absorbed into $M$. Taking the supremum over all $\theta \in [-h,0]$ and using \cref{eq:duality:8} then gives \cref{eq:duality:18}. If $\omega = 0$ (i.e. the semigroup $S$ is bounded) then the factor $e^{\omega\theta}(e^{\omega t} - 1)\omega^{-1}$ must be replaced by $t$, which is just the limiting value of the factor $(e^{\omega t} - 1)\omega^{-1}$ in the final inequality. \qedhere
\end{steps}
\end{proof}

A function $u : \RR_+ \to X$ is called a {\bf solution} of \cref{eq:adde:2} if $u$ is continuous and it satisfies \cref{eq:adde:2}. As a consequence of \cref{eq:duality:18} we have the following result on the existence of a unique global solution of \cref{eq:adde:2}. The proof using Banach's fixed point theorem is entirely standard and will be omitted. 

\begin{corollary}
  \label{cor:duality:1}
  Let $F : X \to Y$ be globally Lipschitz continuous. For every initial condition $\phi \in X$ there exists a unique solution $u \in C(\RR_+,X)$ of \cref{eq:adde:2} with $T_0$ as in \cref{eq:adde:3} and $G \DEF \ell \circ F$.
\end{corollary}

\begin{remark}
  Clearly a weaker, \emph{local} condition on $F$ yields a correspondingly weaker result about the existence of maximal \emph{local} solutions. More generally, here and in \cref{sec:adde:adde_as_aie} below there is no technical impediment to the consideration of local solutions of the nonlinear problems \cref{eq:adde:1,eq:adde:2} along the lines of \cite[Chapter VII]{Diekmann1995}, but purely for simplicity we restrict our attention to global solutions instead. \hfill \QEDEX
\end{remark}

\section{Abstract DDEs as abstract integral equations}\label{sec:adde:adde_as_aie}
Here we explain in detail the relationship between \cref{eq:adde:1} and \cref{eq:adde:2}. Throughout we assume that $T_0$ is the shift semigroup given by \cref{eq:adde:3} and $G \DEF \ell \circ F$ with $F : X \to Y$ as in \cref{eq:adde:1} a continuous function and $\ell$ defined by \cref{eq:duality:19}. 
\begin{definition}
  \label{def:adde_as_aie:1}
  A function $x : [-h,\infty) \to Y$ is called a {\bf classical solution} of \cref{eq:adde:1} if $x$ is continuous on $[-h,\infty)$, continuously differentiable on $\RR_+$, $x(t) \in \DOM(B)$ for all $t \ge 0$ and $x$ satisfies \cref{eq:adde:1}. \hfill \QEDEX
\end{definition}

\begin{lemma}
  \label{lem:adde_as_aie:3}
  Let $x \in C([-h,\infty),Y)$. Then the map $t \mapsto x_t$ is in $C(\RR_+,X)$.
\end{lemma}
\begin{proof}
  Let $t \ge 0$ and $\EPS > 0$ be arbitrary. Then for any $s \ge 0$ with $|t - s| \le 1$ we have $[s-h,s] \subseteq I \DEF [-h,t + 1]$. Of course also $[t - h,t] \subseteq I$. Since $x$ is uniformly continuous on $I$ there exists $\delta \in (0,1)$ such that $|t - s| \le \delta$ implies $\|x(t + \theta) - x(s + \theta)\| \le \EPS$ for all $\theta \in [-h,0]$. Hence if $|t - s| \le \delta$ then
    \begin{displaymath}
      \|x_t - x_s\| = \sup_{\theta \in [-h,0]}{\|x(t+\theta) - x(s+\theta)\|} \le \EPS,
    \end{displaymath}
    which proves continuity of $u$ at $t$.
\end{proof}
\begin{lemma}
  \label{lem:adde_as_aie:1}
  Let $x$ be a classical solution of \cref{eq:adde:1}. Then
  \begin{equation}
    \label{eq:adde_as_aie:2}
    x(t) = S(t)\phi(0) + \int_0^t{S(t - \tau)F(x_{\tau})\,d\tau}, \qquad \forall\,t \ge 0,
  \end{equation}
  where the integral is of the Riemann type.
\end{lemma}
\begin{proof}
  For $t = 0$ the statement is clear. Assume $t > 0$ and define $v : [0,t] \to Y$ by $v(\tau) \DEF S(t - \tau)x(\tau)$. We claim that $v$ is differentiable with derivative
  \begin{equation}
    \label{eq:adde_as_aie:1}
    \dot{v}(\tau) = S(t - \tau)\dot{x}(\tau) - S(t - \tau)Bx(\tau), \qquad \forall\, \tau \in [0,t].
  \end{equation}
  (Of course, this is suggested by a formal application of the product rule.) Indeed, for $\tau \in [0,t]$ and $0 \neq \delta \in \RR$ such that $\tau + \delta \in [0,t]$ we have
  \begin{align*}
    \frac{1}{\delta}[S(t - (\tau + \delta))x(\tau +\delta) - S(t - \tau)x(\tau)] &= S(t - (\tau + \delta)) \frac{1}{\delta}[x(\tau + \delta) - x(\tau)]\\
    &+ \frac{1}{\delta}[S(t - (\tau+\delta)) - S(t - \tau)]x(\tau).
  \end{align*}
  By the strong continuity of $S$ and the differentiability of $x$ the first term on the right converges to $S(t - \tau)\dot{x}(\tau)$ as $\delta \to 0$. Since $x(\tau) \in \DOM(B)$ by assumption we see that the second term tends to $-S(t - \tau)Bx(\tau)$ as $\delta \to 0$. This proves that $v$ is differentiable at $\tau$ and \cref{eq:adde_as_aie:1} holds. Substituting for $\dot{x}$ from \cref{eq:adde:1} then yields
  \begin{displaymath}
    \dot{v}(\tau) = S(t - \tau)Bx(\tau) + S(t - \tau)F(x_\tau) - S(t - \tau)Bx(\tau) = S(t - \tau)F(x_\tau),
  \end{displaymath}
  and $\dot{v}$ is continuous by \cref{lem:adde_as_aie:3}. Hence
  \begin{displaymath}
    x(t) - S(t)\phi(0) = v(t) - v(0) = \int_0^t{\dot{v}(\tau)\,d\tau} = \int_0^t{S(t - \tau)F(x_\tau)\,d\tau},
  \end{displaymath}
  as promised.
\end{proof}
\Cref{def:adde_as_aie:1} is quite restrictive. For example, only initial conditions $\phi \in X$ with $\phi(0) \in \DOM(B)$ are admissible. It is therefore useful to follow \cite{Travis1974} and \cite{Wu1996} in introducing a weaker solution concept, motivated by \cref{lem:adde_as_aie:1}.
\begin{definition}
  \label{def:adde_as_aie:2}
  A function $x \in C([-h,\infty), Y)$ is called a {\bf mild solution} of \cref{eq:adde:1} if $x_0 = \phi$ and $x$ satisfies \cref{eq:adde_as_aie:2}. \hfill \QEDEX
\end{definition}
By \cref{lem:adde_as_aie:3} the continuity of $x$ in the above definition ensures that the Riemann-integral in the right-hand side of \cref{eq:adde_as_aie:2} is well-defined. 
It is legitimate to ask when a mild solution of \cref{eq:adde:1} is also a classical solution. For this we refer to \cite[Proposition 2.3]{Travis1974} and \cite[Theorem 2.1.4 and Remark 2.1.5]{Wu1996}. However, one particular case is of special relevance in view of some of the examples we have in mind \cite{VanGils2013}. 
\begin{proposition}
  \label{prop:adde_as_aie:2}
  Any mild solution of \cref{eq:adde:1} with $B=0$ is a classical solution.
\end{proposition}
\begin{proof}
  Let $x$ be a mild solution. Since $S \equiv I$ (the trivial \C{0}-semigroup consisting of the identity on $Y$) it follows that
  \begin{displaymath}
    x(t) = \phi(0) + \int_0^t{F(x_{\tau})\,d\tau}, \qquad \forall\,t \ge 0.
  \end{displaymath}
  Differentiation with respect to $t$ then yields the result.
\end{proof}
We now give a one-to-one correspondence between solutions of \cref{eq:adde:2} and mild solutions of \cref{eq:adde:1}.

\begin{theorem}
  \label{thm:adde_as_aie:1}
  Let $T_0$ be the shift semigroup given by \cref{eq:adde:3} and define $G \DEF \ell \circ F$ with $F : X \to Y$ continuous and $\ell$ as in \cref{eq:duality:19}. Let $\phi \in X$ be an initial condition. The following statements hold.
  \begin{enumerate}
  \item[\textup{(i)}] 
    Suppose that $x$ is a mild solution of \cref{eq:adde:1}. Define $u : \RR_+ \to X$ by 
    \begin{displaymath}
      u(t) \DEF x_t, \qquad \forall\,t \ge 0.
    \end{displaymath}
    Then $u$ is a solution of \cref{eq:adde:2}.
  \item[\textup{(ii)}] 
    Suppose that $u$ is a solution of \cref{eq:adde:2}. Define $x : [-h,\infty) \to Y$ by
    \begin{displaymath}
      x(t) \DEF
      \begin{cases}
        \phi(t),& -h \le t \le 0,\\
        u(t)(0),&\hphantom{-} 0 \le t.
      \end{cases}
    \end{displaymath}
    Then $x$ is a mild solution of \cref{eq:adde:1}.
  \end{enumerate}
\end{theorem}
\begin{proof}
  We prove (i). If we apply \cref{eq:duality:8} in \cref{prop:duality:1} with $f(\tau) = F(u(\tau))$ then we find for all $t \ge 0$ and $\theta \in [-h,0]$ that
  \begin{displaymath}
    (T_0(t)\phi)(\theta) + j^{-1}\Bigl(\int_0^t{\SUNSTAR{T_0}(t - \tau)\ell F(u(\tau))\,d\tau} \Bigr)(\theta) = \phi(t + \theta) = x(t + \theta) = u(t)(\theta),
  \end{displaymath}
  if $-h \le t + \theta \le 0$ while
  \begin{align*}
    (T_0(t)\phi)(\theta) &+ j^{-1}\Bigl(\int_0^t{\SUNSTAR{T_0}(t - \tau)\ell F(u(\tau))\,d\tau} \Bigr)(\theta)\\
    &=  S(t + \theta)\phi(0) + \int_0^{t + \theta}{S(t - \tau + \theta)F(u(\tau))\,d\tau} = x(t + \theta) = u(t)(\theta),
  \end{align*}
  if $0 \le t + \theta$. Hence $u$ is a solution of \cref{eq:adde:2}.
  
  We prove (ii). 
  \begin{steps}
  \item
    It is clear that $x$ is continuous on $[-h,0]$. Since point evaluation acts continuously on elements of $X$, $x$ is also continuous on $\RR_+$. Since $u(0)(0) = \phi(0)$ it follows that $x$ is continuous on $[-h,\infty)$. We need to show that $x$ satisfies \cref{eq:adde_as_aie:2}. By \cref{prop:duality:1} we find that
    \begin{align}
      x(t) = u(t)(0) &= (T_0(t)\phi)(0) + j^{-1}\Bigl(\int_0^t{\SUNSTAR{T_0}(t - \tau)\ell F(u(\tau))\,d\tau} \Bigr)(0)\nonumber\\
      &= S(t)\phi(0) + \int_0^t{S(t - \tau)F(u(\tau))\,d\tau}.\label{eq:adde_as_aie:3}
    \end{align}
  \item
    By \cref{eq:adde_as_aie:3} we are done once we prove that $u(\tau) = x_{\tau}$ for all $\tau \ge 0$. Let $\theta \in [-h,0]$ be arbitrary. First consider the case $-h \le \tau + \theta \le 0$. Then
    \begin{displaymath}
      x_{\tau}(\theta) = x(\tau + \theta) = \phi(\tau + \theta) = (T_0(\tau)\phi)(\theta) = u(\tau)(\theta),
    \end{displaymath}
    by \cref{prop:duality:1}. Next, consider the case that $0 \le \tau + \theta$. Then
    \begin{align*}
      x_{\tau}(\theta) = x(\tau + \theta) &= u(\tau +\theta)(0)\\
      &= (T_0(\tau + \theta)\phi)(0) + j^{-1}\Bigl(\int_0^{\tau + \theta}{\SUNSTAR{T_0}(\tau + \theta - s)\ell F(u(s))\,ds} \Bigr)(0)\\
      &= S(\tau + \theta)\phi(0) + \int_0^{\tau + \theta}{S(\tau + \theta - s)F(u(s))\,ds}\\
      &= (T_0(\tau)\phi)(\theta) + j^{-1}\Bigl(\int_0^{\tau}{\SUNSTAR{T_0}(\tau - s)\ell F(u(s))\,ds} \Bigr)(\theta)\\
      &= u(\tau)(\theta),
    \end{align*}
    where \cref{prop:duality:1} was used twice. \qedhere
\end{steps}
\end{proof}
In view of \cref{prop:adde_as_aie:2,thm:adde_as_aie:1} there is a one-to-one correspondence between {\it classical} solutions of \cref{eq:adde:1} with $B = 0$ and solutions of \cref{eq:adde:2}. Also, together with \cref{cor:duality:1} we obtain
\begin{corollary}
  \label{cor:adde_as_aie:1}
  Let $F : X \to Y$ be globally Lipschitz continuous. For every initial condition $\phi \in X$ there exists a unique mild solution of \cref{eq:adde:1}.
\end{corollary}


%% file: sunstar.tex
\section{Linear perturbations without sun-reflexivity}\label{sec:sunstar:linear}
Let $T_0$ be a \C{0}-semigroup on a real or complex Banach space and let $A_0$ be its generator. Clearly we have in mind the shift semigroup from \cref{eq:adde:3} but the results in this section have a general validity. We are interested in perturbations of $T_0$ by bounded linear operators on $X$ with values in the larger space $\SUNSTAR{X}$ in which $X$ lies embedded by $j : X \to \SUNSTAR{X}$. For the case that $X$ is sun-reflexive with respect to $T_0$ this problem is discussed in detail in \cite{Clement1987} and \cite[Section III.2]{Diekmann1995}. However, the assumption of sun-reflexivity may be too restrictive, as shown by the class \cref{eq:adde:1} of abstract DDEs. So, we reformulate a subset of the results from \cite{Clement1987,Diekmann1995} without assuming sun-reflexivity. 

The lack of sun-reflexivity is felt immediately when we review the first result of \cite[\S III.2]{Diekmann1995}. Its proof still applies verbatim, except for the final step. Namely, without sun-reflexivity all we can say is that $w$ takes values in $\SUNSUN{X}$. 
\begin{lemma}[cf. {\cite[Lemma III.2.1]{Diekmann1995}}] 
  \label{lem:sunstar:5}
  Let $f : \RR_+ \to \SUNSTAR{X}$ be continuous. Define the set
  \[
    \Omega \DEF \{(t,s,r) \in \RR^3\,:\,0 \le r \le s \le t\},
  \]
  and define $w : \Omega \to \SUNSTAR{X}$ by
  \begin{equation}
    \label{eq:sunstar:19}
    w(t,s,r) \DEF \int_r^s{\SUNSTAR{T_0}(t - \tau)f(\tau)\,d\tau}.
  \end{equation}
  Then $w$ is continuous and takes values in $\SUNSUN{X}$.
\end{lemma}

In order to proceed, we therefore impose the following condition as a weaker substitute for the hypothesis of sun-reflexivity of $X$ with respect to $T_0$. 
\vskip 1ex
\newcommand{\HNULL}{\hyperref[H0]{\textup{(H0)}}}
\label{H0}
\noindent {\bf (H0)} The semigroup $T_0$ and the operator $L \in \BND(X,\SUNSTAR{X})$ have the property that
\begin{displaymath}
  \int_0^t{\SUNSTAR{T_0}(t - \tau)Lu(\tau)\,d\tau} \in j(X),
\end{displaymath}
for all continuous functions $u : \RR_+ \to X$ and all $t \ge 0$. \hfill \QEDEX
\vskip 1ex

We note that \cref{lem:sunstar:5} implies that {\HNULL} is satisfied when $X$ is sun-reflexive with respect to $T_0$. More generally, most results from \cite[\S III.2]{Diekmann1995} continue to hold upon imposing {\HNULL}. In particular, perturbation of $T_0$ by a compliant operator $L \in \BND(X,\SUNSTAR{X})$ yields a perturbed \C{0}-semigroup on $X$. Namely,
\begin{theorem}[cf. {\cite[Theorem III.2.4]{Diekmann1995}}]
  \label{thm:sunstar:1}
  If $T_0$ and $L$ satisfy {\HNULL} then there exists a unique \C{0}-semigroup $T$ on $X$ such that
  \begin{equation}
    \label{eq:sunstar:1}
    T(t)\phi = T_0(t)\phi + j^{-1}\int_0^t{\SUNSTAR{T_0}(t - \tau)LT(\tau)\phi\,d\tau},
  \end{equation}
  for all $\phi \in X$ and for all $t \ge 0$.
\end{theorem}
We now wish to understand how the relationship \cref{eq:sunstar:1} between the semigroups $T_0$ and $T$ is reflected on the level of their generators $A_0$ and $A$. First we summarize what stays the same.

\begin{proposition}
  Suppose that {\HNULL} holds. Then $\SUN{X}$ is also the maximal subspace of strong continuity of the adjoint semigroup $\STAR{T}$ on $\STAR{X}$. The {\WSTAR} generator $\STAR{A}$ of $\STAR{T}$ is given by
  \[
    \DOM(\STAR{A}) = \DOM(\STAR{A_0}), \qquad \STAR{A} = \STAR{A_0} + \STAR{L},
  \]
  and the generator $\SUN{A}$ of the restriction $\SUN{T}$ of $\STAR{T}$ to $\SUN{X}$ is
  \[
    \DOM(\SUN{A}) = \{\SUN{\phi} \in \DOM(\STAR{A_0})\,:\, (\STAR{A_0} + \STAR{L})\SUN{\phi} \in \SUN{X}\}, \qquad \SUN{A} = \STAR{A_0} + \STAR{L}.
  \]
  $\SUNSUN{X}$ is also the maximal subspace of strong continuity of the adjoint semigroup $\SUNSTAR{T}$.
\end{proposition}
\begin{proof}
  The proofs of \cite[Lemmas III.2.6 and III.2.7, Corollaries III.2.8 and III.2.9]{Diekmann1995} still apply verbatim, given that {\HNULL} is fulfilled.
\end{proof}

It remains to characterize the {\WSTAR} generator $\SUNSTAR{A}$ of the adjoint semigroup $\SUNSTAR{T}$. In the sun-reflexive case we have $\DOM(\SUNSTAR{A}) = \DOM(\SUNSTAR{A_0})$ and $\SUNSTAR{A} = \SUNSTAR{A_0} + Lj^{-1}$ but this cannot be true without sun-reflexivity. Namely, suppose that $\DOM(\SUNSTAR{A_0}) \subseteq j(X)$. Since $j(X)$ is closed it follows that $\SUNSUN{X} = \BAR{\DOM(\SUNSTAR{A_0})} \subseteq j(X)$ so $X$ is sun-reflexive with respect to $T_0$. So, we restrict to elements in the domains of $\SUNSTAR{A_0}$ and $\SUNSTAR{A}$ that lie in the range of $j$.

\begin{lemma}[cf. {\cite[Lemma III.2.11]{Diekmann1995}}]
  \label{lem:sunstar:2}
  Define $U : \RR_+ \to \BND(X)$ as $U(t) \DEF T_0(t) - T(t)$. For every $\phi \in X$ it holds that $t^{-1}\SUNSTAR{U}(t)j\phi \to L\phi$ {\WSTARLY} in $\SUNSTAR{X}$ as $t \downarrow 0$.
\end{lemma}
\begin{proof}
  For every $\SUN{\phi} \in \SUN{X}$ and every $t > 0$ we have
  \[
    \frac{1}{t}\PAIR{\SUN{\phi}}{\SUNSTAR{U}(t)j\phi} = \frac{1}{t}\PAIR{\SUN{U}(t)\SUN{\phi}}{j\phi} = \frac{1}{t}\PAIR{\phi}{\STAR{U}(t)\SUN{\phi}},
  \]
  which tends to $\PAIR{\phi}{\STAR{L}\SUN{\phi}}$ $= \PAIR{\SUN{\phi}}{L\phi}$ as $t \downarrow 0$ by \cite[Lemma III.2.7]{Diekmann1995}.
\end{proof}

\begin{proposition}[cf. {\cite[Corollary III.2.12]{Diekmann1995}}]
  \label{prop:sunstar:8}
  For the {\WSTAR} generators $\SUNSTAR{A_0}$ and $\SUNSTAR{A}$ it holds that
  \begin{equation}
    \label{eq:sunstar:24}
    \DOM(\SUNSTAR{A}) \cap j(X) = \DOM(\SUNSTAR{A_0}) \cap j(X),
  \end{equation} 
  and $\SUNSTAR{A} = \SUNSTAR{A_0} + Lj^{-1}$ on this subspace.
\end{proposition}
\begin{proof}
  From the general theory we know that
  \begin{equation}
    \label{eq:sunstar:6}
    \begin{aligned}
      \DOM(\SUNSTAR{A_0}) &= \{\SUNSUN{\phi} \in \SUNSUN{X}\,:\, \frac{1}{t}(\SUNSTAR{T_0}(t)\SUNSUN{\phi} - \SUNSUN{\phi}) \text{ converges {\WSTARLY} in } \SUNSTAR{X}\},\\
      \SUNSTAR{A_0}\SUNSUN{\phi} &= \WLIM_{t \downarrow 0}{\frac{1}{t}(\SUNSTAR{T_0}(t)\SUNSUN{\phi} - \SUNSUN{\phi})},
    \end{aligned}
  \end{equation}
  where $\WLIM$ is a {\WSTAR}-limit in $\SUNSTAR{X}$, with an analogous result for $\SUNSTAR{A}$. Also, let $U$ be as in \cref{lem:sunstar:2}. For any $\SUNSUN{\phi} \in \DOM(\SUNSTAR{A_0}) \cap j(X)$ and any $\phi \DEF j^{-1}\SUNSUN{\phi} \in X$,
  \[
    \frac{1}{t}(\SUNSTAR{T}(t)\SUNSUN{\phi} - \SUNSUN{\phi}) = \frac{1}{t}(\SUNSTAR{T_0}(t)\SUNSUN{\phi} - \SUNSUN{\phi}) + \frac{1}{t}\SUNSTAR{U}(t)j\phi \to \SUNSTAR{A_0}\SUNSUN{\phi} + Lj^{-1}\SUNSUN{\phi},
  \]
  {\WSTARLY} in $\SUNSTAR{X}$ as $t \downarrow 0$ by \cref{eq:sunstar:6} and \cref{lem:sunstar:2}. It follows that $\SUNSUN{\phi} \in \DOM(\SUNSTAR{A})$ and $\SUNSTAR{A}\SUNSUN{\phi} = \SUNSTAR{A_0}\SUNSUN{\phi} + Lj^{-1}\SUNSUN{\phi}$. By interchanging the roles of $T_0$ and $T$ we obtain the reverse inclusion.
\end{proof}

In the sun-reflexive case, the intersections in \cref{eq:sunstar:24} simply reproduce $\DOM(\SUNSTAR{A})$ and $\DOM(\SUNSTAR{A_0})$ because then both domains are contained in the range of $j$.

\begin{corollary}[cf. {\cite[Corollary III.2.13]{Diekmann1995}}]
  \label{cor:sunstar:1}
  The generator $A$ is given by
  \begin{align}
    \DOM(A) &= \{\phi \in X\,:\,j\phi \in \DOM(\SUNSTAR{A_0}) \text{ and } \SUNSTAR{A_0}j\phi + L\phi \in j(X)\},\label{eq:sunstar:9}\\
    A\phi &= j^{-1}(\SUNSTAR{A_0}j\phi + L\phi).\nonumber
  \end{align}
\end{corollary}
\begin{proof}
  The generator $\SUNSUN{A}$ is the part of $\SUNSTAR{A}$ in $\SUNSUN{X}$,
  \[
    \DOM(\SUNSUN{A}) = \{\SUNSUN{\phi} \in \DOM(\SUNSTAR{A}) \,:\, \SUNSTAR{A}\SUNSUN{\phi} \in \SUNSUN{X}\}, \qquad \SUNSUN{A}\SUNSUN{\phi} = \SUNSTAR{A}\SUNSUN{\phi}.
  \]  
  Also, for every $\phi \in X$ and every $t > 0$,
  \begin{equation}
    \label{eq:sunstar:10}
    j\Bigl(\frac{1}{t}(T(t)\phi - \phi)\Bigr) = \frac{1}{t}(\SUNSUN{T}(t)j\phi - j\phi).
  \end{equation}
  Suppose that $\phi$ is in the right-hand side of \cref{eq:sunstar:9}. Now $j\phi \in \DOM(\SUNSTAR{A_0}) \cap j(X)$ so $j\phi \in \DOM(\SUNSTAR{A})$ by \cref{prop:sunstar:8} and
  \begin{displaymath}
    \SUNSTAR{A}j\phi = \SUNSTAR{A_0}j\phi + L \phi \in j(X) \subseteq \SUNSUN{X}.
  \end{displaymath}
  Consequently $j\phi \in \DOM(\SUNSUN{A})$ and the right-hand side of \cref{eq:sunstar:10} converges in norm to $\SUNSTAR{A}j\phi = \SUNSTAR{A_0}j\phi + L\phi$ as $t \downarrow 0$. By continuity of $j^{-1}$ this implies that $\phi \in \DOM(A)$ and 
  \begin{displaymath}
    A\phi = j^{-1}(\SUNSTAR{A_0}j\phi + L\phi).
  \end{displaymath}
  Conversely, suppose $\phi \in \DOM(A)$. Then by \cref{eq:sunstar:10} and the continuity of $j$ we have that $j\phi \in \DOM(\SUNSUN{A})$ and $\SUNSUN{A}j\phi = j A\phi$. \Cref{prop:sunstar:8} implies that $j\phi \in \DOM(\SUNSTAR{A_0})$ and
  \begin{displaymath}
    \SUNSTAR{A_0}j\phi + L\phi  = \SUNSTAR{A}j\phi = \SUNSUN{A}j\phi = jA\phi \in j(X),
  \end{displaymath}
  so $\phi$ is in the right-hand side of \cref{eq:sunstar:9}.
\end{proof}

\section{Comments on perturbations for abstract DDEs}\label{sec:sunstar:linadde}
A non-trivial example of a semigroup $T_0$ and an operator $L$ satisfying {\HNULL} from \cref{sec:sunstar:linear} arises when we let $T_0$ be the shift semigroup from \cref{eq:adde:3} and consider linear abstract DDEs. The corresponding linear initial value problem is
\begin{subequations}
  \label{eq:sunstar:linadde}
  \begin{alignat}{2}
    \dot{x}(t) &= B x(t) + \Phi x_t, &\qquad &t \ge 0,\\
    x(\theta) &= \phi(\theta), &&\theta \in [-h,0],
  \end{alignat}  
\end{subequations}
with $\Phi \in \BND(X,Y)$. (For instance, it could be obtained from \cref{eq:adde:1} by linearization at a zero equilibrium solution, assuming $F$ is of class $C^1$ on a neighborhood of $0 \in X$, in which case $\Phi = DF(0)$ is the Fr\'echet derivative of $F$ at the origin.) If we define $L \DEF \ell \circ \Phi \in \BND(X,\SUNSTAR{X})$ then \Cref{prop:duality:1} implies that
\[
  \int_0^t{\SUNSTAR{T_0}(t - \tau)L u(\tau)\,d\tau} \in j(X),
\]
for every continuous function $u : \RR_+ \to X$, so {\HNULL} is satisfied. \Cref{thm:sunstar:1} then implies the existence of a \C{0}-semigroup $T$ obtained by perturbing $T_0$ by $L$ that is the unique solution of \cref{eq:sunstar:1}. \cref{thm:adde_as_aie:1} with $G = L$ gives the precise correspondence between solutions of \cref{eq:sunstar:1} and \cref{eq:sunstar:linadde}.

Having explained how linear abstract DDEs fit into the general context of \cref{sec:sunstar:linear}, in conclusion we turn our attention again to the general, nonlinear case. As is well-known, for \emph{classical} DDEs the operator $G$ in \cref{eq:adde:2} is of finite rank. This has some far-reaching consequences, particularly for the spectral theory. For \emph{abstract} DDEs we note that \Cref{thm:duality:1} implies that
\begin{equation}
  \label{eq:sunstar:xsunstar}
  \SUNSTAR{X} \simeq \SUNSTAR{Y} \times \STAR{[\LP{1}([0,h],\STAR{Y})]},
\end{equation}
also see \cref{fig:duality:1}. So, as could be expected, the finite rank property is lost as soon as $Y$ is infinite-dimensional, but \cref{lem:duality:8} from \cref{sec:duality:range} \emph{does} imply that
\begin{equation}
  \label{eq:sunstar:38}
   G(\phi) = (j_Y F(\phi), 0), \qquad \forall\,\phi \in X.
\end{equation}
Therefore, as in the classical case, it remains true that $G$ takes nonzero values only in the first component of $\SUNSTAR{X}$. In general it is impossible to obtain an explicit representation for the dual space $\STAR{[\LP{1}([0,h],\STAR{Y})]}$ appearing in \cref{eq:sunstar:xsunstar} unless $\STARSTAR{Y}$ has the Radon-Nikod\'ym property \cite[Theorem IV.1]{Diestel1977}. However, $\SUNSTAR{X}$ contains a familiar, nontrivial closed subspace that itself contains the range of $G$. Here we give an elementary proof of this.

\begin{proposition}
  Let $I$ be any non-trivial interval. Define $V \DEF \LP{\infty}(I,Y)$ and $W \DEF \LP{1}(I,\STAR{Y})$. Then $V$ is isometrically isomorphic to a closed subspace of $\STAR{W}$.
\end{proposition}
\begin{proof}
  We will show that $\kappa : V \to \STAR{W}$ is an isometric embedding, with $\kappa$ defined by
\begin{displaymath}
  \PAIR{w}{\kappa v} \DEF \int_{I}{\PAIR{v(s)}{w(s)}\,ds}, \qquad \forall\,v \in V,\, w \in W.
\end{displaymath}
The right-hand side is well-defined as an ordinary Lebesgue integral. In fact,
\begin{displaymath}
  |\PAIR{w}{\kappa v}| \le \int_{I}{|\PAIR{v(s)}{w(s)}|\,ds} \le \|v\|_{\infty} \int_{I}{\|w(-s)\|_Y\,ds} = \|w\|_1 \cdot \|v\|_{\infty}.
\end{displaymath}
Taking the supremum over all $w \in W$ with $\|w\|_1 = 1$ it follows that 
\begin{displaymath}
  \|\kappa v\|_{\STAR{W}} \le \|v\|_{\infty}, \qquad \forall\,v \in V.
\end{displaymath}
For the other inequality we adapt the argument for the scalar case \cite[Theorem 37.10]{Aliprantis1998}.
\begin{steps}
\item
  Let $v \in V$ be countably valued, i.e. there exists a sequence $(I_i)_{i \in \NN}$ of measurable subsets of $I$ and a sequence $(y_i)_{i \in \NN}$ of nonzero elements of $Y$ such that
  \begin{displaymath}
    v(s) = \sum_{i=1}^{\infty}{\chi_i(s)y_i}, \qquad \text{a.e. } s \in I,
  \end{displaymath}
  where $\chi_i$ is the indicator function of $I_i$. The Hahn-Banach theorem implies that for each $i \in \NN$ there exists $\STAR{y}_i \in \STAR{Y}$ such that
  \begin{displaymath}
    \|\STAR{y}_i\|_{\STAR{Y}} = 1, \qquad \PAIR{y_i}{\STAR{y}_i} = \|y_i\|_Y.
  \end{displaymath}
  Define $w \in W$ by 
  \begin{displaymath}
    w(s) \DEF \sum_{i=1}^{\infty}{\chi_i(s)\STAR{y}_i}, \qquad \forall\,s \in I,
  \end{displaymath}
  where it does not matter which $\STAR{y}_i$ we choose. Let $\EPS > 0$ and define 
  \begin{displaymath}
    A_{\EPS} \DEF \{s \in I\,:\, \|\kappa v\|_{\STAR{W}} + \EPS < \|v(s)\|_Y\}.
  \end{displaymath}
  Denoting by $|A_{\EPS}|$ the measure of $A_{\EPS}$, we then have the estimate
  \begin{align*}
    (\|\kappa v\|_{\STAR{W}} + \EPS)|A_{\EPS}| &\le \int_{A_{\EPS}}{\|v(s)\|_Y\,ds} = \int_{A_{\EPS}}{\PAIR{v(s)}{w(s)}\,ds}\\
    &\le \|\kappa v\|_{\STAR{W}} \|\chi_{A_{\EPS}} w\|_1 = \|\kappa v\|_{\STAR{W}} \int_{A_{\EPS}}{\Bigl\|\sum_{i=1}^{\infty}{\chi_i(s)\STAR{y}_i}\Bigr\|\,ds}\\
    &= \|\kappa v\|_{\STAR{W}} \int_{A_{\EPS}}{\sum_{i=1}^{\infty}{\chi_i(s)}\,ds} \le \|\kappa v\|_{\STAR{W}}|A_{\EPS}|.
  \end{align*}
  This implies that $|A_{\EPS}| = 0$, so $\|v\|_{\infty} \le \|\kappa v\|_{\STAR{W}} + \EPS$. Since $\EPS$ was arbitrary, it follows that
  \begin{equation}
    \label{eq:notes:1}
    \|v\|_{\infty} \le \|\kappa v\|_{\STAR{W}}.
  \end{equation}
\item
  Let $v \in V$ be arbitrary. By corollary of Pettis' measurability theorem \cite[Corollary 1.1.2]{Arendt2011} $v$ is the \emph{uniform} limit a.e. of a sequence of measurable countably valued functions. By taking limits and using the boundedness of $\kappa$ we see that \eqref{eq:notes:1} also holds for $v$. \hfill \qedhere
\end{steps}
\end{proof}

Applying the above result to our situation, we conclude that $\SUNSTAR{X}$ contains a closed subspace isometrically isomorphic to $\SUNSTAR{Y} \times \LP{\infty}([-h,0],Y)$. This property together with \cref{eq:sunstar:38} is useful for explicit calculations at a later stage, for example when computing local normal forms \cite{VanGils2013}.


%% file: conclusion.tex
\section{Conclusion and outlook}\label{sec:adde:conclusion}
In this article we did not make any assumptions beyond strong continuity about the regularity of the semigroup $S$ generated by $B$ in \cref{eq:adde:1}. As an advantage, there was no need - at least at the abstract level - to distinguish between the cases $B \neq 0$ and $B = 0$. We note that the latter case is \emph{not} admissible if $Y$ is infinite-dimensional and $S$ is required to be immediately or eventually compact. Indeed, such a requirement often occurs in \cite{Travis1974,Wu1996,Faria2002,Faria2006}. 

However, mere strong continuity is insufficient for most questions from local  dynamical systems theory, because of the absence of a general spectral mapping principle \cite[\S IV.3a]{Engel2000}. Therefore, in order to proceed we require $S$ to be immediately norm continuous. This includes the cases that $S$ is analytic or immediately compact. Importantly, immediate norm continuity of $S$ implies eventual norm continuity of the \C{0}-semigroup $T$ corresponding to the linear initial value problem \cref{eq:sunstar:linadde} \cite[Theorem VI.6.6]{Engel2000}. Moreover, for such semigroups there exists a general spectral mapping principle and as a consequence the growth bound of $T$ can be deduced from its spectral bound.

Nonlinear perturbation and the existence and properties of local invariant manifolds are the subject of an article that is currently in preparation. As is true for the present article, part of the underlying work had already been completed not long after publication of \cite{VanGils2013}. We intend to present in two stages: First we prove results for \cref{eq:adde:2}, generalizing a small subset of the material from \cite{Clement1987, Clement1988, Clement1989, Clement1989b, Diekmann1991,Diekmann1995} and along the way formulating suitable conditions in the spirit of {\HNULL}. By virtue of \cref{thm:adde_as_aie:1} applicability of these results to \cref{eq:adde:1} will then follow as a special case. Although slightly less direct, this approach has the benefit that the work will also apply to other problems that can be cast in the form \cref{eq:adde:2}, such as those studied in \cite{Diekmann2008}.

\subsection*{Acknowledgements}
I would like to thank Prof. Odo Diekmann for his practical help and encouragement in turning old notes into a - hopefully readable - public article, and for contributing much to the considerable body of theory from which I have the pleasure to benefit in my research. I would also like to thank Prof. Stephan A. van Gils for his contributions to the field of analysis of delay equations, and for involving me in the motivating work leading to \cite{VanGils2013}. Finally, I would like to thank Prof. Stephan A. van Gils and Prof. Yuri A. Kuznetsov for engaging discussions on various occasions, for their interest and their valuable comments.


%% file: integration.tex
\section{Bounded variation and integration}\label{sec:integration}
This appendix is a vector-valued counterpart to \cite[Appendix I]{Diekmann1995}. After introducing vector-valued functions of bounded variation in \cref{sec:integration:bv} we discuss in \cref{sec:integration:rs} how the traditional Riemann-Stieltjes integral involving an integrand $f: [a,b] \to \RR$ and a function $\eta : [a,b] \to \RR$ of bounded variation can be extended to the case that \emph{both} $f$ and $\eta$ take values in Banach spaces related by a continuous bilinear form\footnote{Alternatively, it is possible to formulate the results from \cref{sec:duality:sun_dual,sec:duality:range,sec:adde:adde_as_aie} in terms of vector measures \cite{Bartle1956,Diestel1977}. However, in the theory of delay equations one traditionally works with functions of bounded variation.}. To our knowledge, such an extension was first performed by Gowurin\footnote{A student from the Saint Petersburg school of G.M. Fichtenholz, Goruwin seems to be better known in the Western literature as M.K. {\it Gavurin} (1911 - 1992). He worked primarily on real, functional and numerical analysis.} \cite{Gowurin1936}. Most results are rather easy generalizations of their scalar counterparts, but there are a few differences as well. Some of this material is also presented in \cite{Gowurin1936} or \cite[Chapter X]{Lang1993}. 


\subsection{Vector-valued functions of bounded variation}\label{sec:integration:bv}
Let $W$ be a Banach space over $\KK \in \{\RR,\CC\}$ and let $f: [a,b] \subset \RR \to W$ be a function. The {\bf total variation function $V_a(f): [a,b] \to [0,\infty]$} is defined by $V_a(f)(a) \DEF 0$ and
\begin{equation}
  \label{eq:integration:V}
  V_a(f)(t) \DEF \sup_{P(a,t)}{\sum_{j=1}^N{\|f(\sigma_j) - f(\sigma_{j-1})\|}}, \qquad \forall\,t \in (a,b],
\end{equation}
where $P(a,t) = ([\sigma_{i-1},\sigma_i])_{i=1}^N$ is a {\bf partition} of $[a,t]$,
\begin{displaymath}
  a = \sigma_0 < \sigma_1 < \ldots < \sigma_N = t.
\end{displaymath}
The norm $\|\cdot\|$ appearing in the sum is the norm of the space $W$. When $V_a(f)(b) < \infty$ we say that $f$ is of {\bf bounded variation}. The first lemma implies that $V_a(f)$ is non-decreasing on $[a,b]$.
\begin{lemma}
  \label{lem:integration:1}
  Let $c \in [a,b]$. Then $V_a(f)(b) = V_a(f)(c) + V_c(f)(b)$.
\end{lemma}
\begin{proof}
  Let $P(a,b) = ([\sigma_{i-1},\sigma_i])_{i=1}^N$ be a partition of $[a,b]$. There exists an index $j_0 \in \{1,\ldots,N\}$ such that $\sigma_{j_0 - 1} \le c \le \sigma_{j_0}$. Hence
  \begin{align*}
    \sum_{j=1}^N{\|f(\sigma_j)} - f(\sigma_{j-1})\| &= \sum_{j \neq j_0}{\|f(\sigma_j) - f(\sigma_{j-1})\|} + \|f(\sigma_{j_0}) - f(c) + f(c) - f(\sigma_{j_0 - 1})\|\\
                                                    &\le \sum_{j \neq j_0}{\|f(\sigma_j) - f(\sigma_{j-1})\|} + \|f(c) - f(\sigma_{j_0 - 1})\| + \|f(\sigma_{j_0}) - f(c)\|\\
                                                    &\le  V_a(f)(c) + V_c(f)(b).
  \end{align*}
  For the reverse inequality we may without loss of generality assume that $c \in (a,b)$. Let $P(a,c) = ([\sigma_{i-1},\sigma_i])_{i=1}^N$ be a partition of $[a,c]$ and let $Q(c,b) = ([\tau_{i-1},\tau_i])_{i=1}^M$ be a partition of $[c,b]$. Then the points
  \begin{displaymath}
    a = \sigma_0 < \sigma_1 < \ldots < \sigma_N = c = \tau_0 < \tau_1 < \ldots < \tau_M = b
  \end{displaymath}
  determine a partition of $[a,b]$. Hence
  \begin{align*}
    \sum_{i=1}^N{\|f(\sigma_i) - f(\sigma_{i-1})  \|} + \sum_{j=1}^M{\|f(\tau_j) - f(\tau_{j-1})  \|} \le V_a(f)(b),
  \end{align*}
  and the result follows by taking the supremum over all such $P$ and $Q$.
\end{proof}
\begin{corollary}
  \label{cor:integration:1}
  If $f : [a,b] \to W$ is of bounded variation, then $f$ is bounded.
\end{corollary}
\begin{proof}
  Suppose that $f$ is not bounded, then there exists a sequence $(t_n)_n$ of points in $[a,b]$ such that $\|f(t_n) - f(a)\| \ge n$ for all $n \in \NN$. On the other hand,
  \begin{displaymath}
    V_a(f)(b) \ge V_a(f)(t_n) \ge \|f(t_n) - f(a)\| \qquad \forall\,n \in \NN
  \end{displaymath}
  showing that $f$ cannot be of bounded variation.
\end{proof}
Verification of the following result is trivial.
\begin{lemma}\label{lem:integration:2}  
  For all $f,g: [a,b] \to W$ of bounded variation and all $\alpha \in \KK$,
  \begin{enumerate}
  \item[\textup{(i)}] 
    $V_a(f + g)(t) \le V_a(f)(t) + V_a(g)(t)$ for all $t \in [a,b]$,
  \item[\textup{(ii)}]
    $V_a(\alpha f) = |\alpha| V_a(f)$.
  \end{enumerate}
  Hence the set $\BV([a,b],W)$ of $W$-valued functions of bounded variation with pointwise addition and scalar multiplication is a vector space over $\KK$.
\end{lemma}
We observe that the product of two elements of $\BV([a,b],W)$ is not defined, unless $W$ admits an appropriate multiplication. We also note that the decomposition theorem for the case that $W = \RR$ that is often used to prove properties of functions of bounded variation (see e.g. \cite[Theorem 1.4 in Appendix I]{Diekmann1995}) does not have a direct analogue, unless we endow $W$ with an order structure. Instead, we exploit the fact that the continuity properties of $f \in \BV([a,b],W)$ are closely related to the continuity properties of the associated variation function, as the next lemma shows.
\begin{lemma}
  \label{lem:integration:3}
  Let $f \in \BV([a,b],W)$.
  \begin{enumerate}
  \item[\textup{(i)}]
    $f$ is left-continuous at $c \in (a,b]$ if and only if $V_a(f)$ is left-continuous at $c$.
  \item[\textup{(ii)}]
    $f$ is right-continuous at $c \in [a,b)$ if and only if $V_a(f)$ is right-continuous at $c$.
  \item[\textup{(iii)}]
    $f$ is continuous at $c \in (a,b)$ if and only if $V_a(f)$ is continuous at $c$.
  \end{enumerate}
\end{lemma}
\begin{proof}
  We prove only the first statement. Statement (ii) is proved analogously and statement (iii) follows from the first two since $f$ is continuous at $c \in (a,b)$ if and only if it is both left- and right-continuous there.
  \par
  \emph{Left-continuity of $V_a(f)$ implies left-continuity of $f$.} For $t \in [a,c]$ we have
  \begin{displaymath}
    \|f(c) - f(t)\| \le V_t(f)(c) = V_a(f)(c) - V_a(f)(t) \to 0 \text{ as } t \uparrow c,
  \end{displaymath}
  where \cref{lem:integration:1} was used in the equality. 
  \par
  \emph{Left-continuity of $f$ implies left-continuity of $V_a(f)$.} Let $\EPS > 0$ be given and let $P(a,c) = ([\sigma_{i-1},\sigma_i])_{i=1}^N$ be a partition of $[a,c]$ such that 
  \begin{equation}
    \label{eq:integration:12}
    \sum_{j=1}^N{\|f(\sigma_j) - f(\sigma_{j-1})\|} \ge V_a(f)(c) - \frac{\EPS}{2}.
  \end{equation}
  By left-continuity of $f$ at $c$ there exists $\delta > 0$ such that $c - \delta > \sigma_{N-1}$ and $t \in (c - \delta,c)$ implies
  \begin{equation}
    \label{eq:integration:11}
    \|f(c) - f(t)\| \le \frac{\EPS}{2}.
  \end{equation}
  By the triangle inequality it holds that
  \begin{displaymath}
    \|f(t) - f(\sigma_{N-1})\| + \|f(c) - f(t)\| \ge \|f(c) - f(\sigma_{N-1})\|,
  \end{displaymath}
  and together with \cref{eq:integration:12} this implies that
  \begin{displaymath}
    \sum_{j=1}^{N-1}{\|f(\sigma_j) - f(\sigma_{j-1})\|} + \|f(t) - f(\sigma_{N-1})\| + \|f(c) - f(t)\| \ge V_a(f)(c) - \frac{\EPS}{2}.
  \end{displaymath}
  Therefore, by \cref{eq:integration:11}, 
  \begin{displaymath}
     \sum_{j=1}^{N-1}{\|f(\sigma_j) - f(\sigma_{j-1})\|} + \|f(t) - f(\sigma_{N-1})\| \ge V_a(f)(c) - \EPS.
  \end{displaymath}
  Since the points $a = \sigma_0 < \sigma_1 < \ldots < \sigma_{N-1} < t$ determine a partition of $[a,t]$ it follows that $V_a(f)(t) \ge V_a(f)(c) - \EPS$. By monotonicity this concludes the proof.
\end{proof}
The following is a direct consequence of \cref{lem:integration:1,lem:integration:3}.
\begin{corollary}
  \label{cor:integration:2}
  The set of discontinuity points of $f \in \BV([a,b],W)$ is countable.
\end{corollary}
For the next result we recall that $\LP{1}([a,b],W)$ denotes the Banach space of Bochner integrable $W$-valued functions on $[a,b]$, see \cite[Appendix C]{Engel2000} or \cite[\S 1.1]{Arendt2011}.
\begin{definition}
  \label{def:integration:4}
  The function $\chi_a \in \BV([a,b],\KK)$ denotes the characteristic function of $(a,b]$ defined by $\chi_a(a) \DEF 0$ and $\chi_a(t) \DEF 1$ for all $t \in (a,b]$. \hfill \QEDEX
\end{definition}
\begin{lemma}
  \label{lem:duality:1}
  Let $f : [a,b] \to W$ and suppose there exist $w \in W$ and $g \in \LP{1}([a,b],W)$ such that
  \begin{equation}
    \label{eq:duality:6}
    f(t) = \chi_a(t)w + \int_a^t{g(s)\,ds}, \qquad \forall\,t \in [a,b].
  \end{equation}
  Then $f \in \BV([a,b],W)$ and $V_a(f)(b) = \|w\| + \|g\|$.
\end{lemma}
\begin{proof}
  Define $\gamma : [a,b] \to W$ by
  \begin{displaymath}
    \gamma(t) \DEF \int_a^t{g(s)\,ds}, \quad \forall t \in [a,b].
  \end{displaymath}
  We clearly have $V_a(\chi_aw)(b) = \|w\|$. Also, by \cite[Proposition 1.2.2d]{Arendt2011} we see that $\gamma \in \BV([a,b],W)$ and $V_a(\gamma)(b) = \|g\|$. Hence \cref{lem:integration:2} implies that
  \begin{equation}
    \label{eq:integration:9}
    V_a(f)(b) \le V_a(\chi_a)(b) + V_a(\gamma)(b) = \|w\| + \|g\|,
  \end{equation}
  so in particular $f \in \BV([a,b],W)$.
  \par
  To see that \cref{eq:integration:9} is in fact an equality, as usual let $P(a,b) = ([\sigma_{i-1},\sigma_i])_{i=1}^N$ be an arbitrary partition of $[a,b]$ and suppose $a < \tau < \sigma_1$. Then 
  \begin{align*}
    V_a(f)(b) &\ge \|f(\tau) - f(a)\| + \|f(\sigma_1) - f(\tau)\| + \sum_{j=2}^N{\|f(\sigma_j) - f(\sigma_{j-1})\|}\\
    &= \|w + \gamma(\tau)\| + \|\gamma(\sigma_1) - \gamma(\tau)\| + \sum_{j=2}^N{\|\gamma(\sigma_j) - \gamma(\sigma_{j-1})\|},
  \end{align*}
  where the sums are understood to be zero in case $N = 1$. Letting $\tau \downarrow 0$ and using the continuity of $\gamma$, we find that
  \begin{displaymath}
    V_a(f)(b) \ge \|w\| + \sum_{j=1}^N{\|\gamma(\sigma_j) - \gamma(\sigma_{j-1})\|}.
  \end{displaymath}
  Upon taking the supremum over all partitions $P(a,b)$ we obtain \cref{eq:integration:9} with equality.
\end{proof}

\subsection{Bilinear Riemann-Stieltjes integration}\label{sec:integration:rs}
Let $V$, $W$ and $Z$ be Banach spaces over $\KK \in \{\RR, \CC\}$. We assume the existence of a continuous bilinear product $\cdot$ on $V \times W$ with values in $Z$,
\begin{displaymath}
  V \times W \ni (v,w) \mapsto v \cdot w \in Z. 
\end{displaymath}
For example, we may take $W = \STAR{V}$ and $Z = \KK$ and define $v \cdot \STAR{v} \DEF \PAIR{v}{\STAR{v}}$, so the product is just the duality pairing between $V$ and $\STAR{V}$. Note that in contrast with \cite[\S 1]{Gowurin1936} we do not assume that the product is symmetric.
\par
Fix an interval $[a,b]$. We define a {\bf tagged partition} $P$ of the interval $[a,b]$ to be a partition $([\sigma_{j-1},\sigma_j])_{j=1}^N$ of $[a,b]$ together with a finite sequence of {\bf sample points} $\tau_j \in [\sigma_{j-1},\sigma_j]$ for $j = 1,\ldots,N$. This will be denoted by $P = ([\sigma_{j-1},\sigma_j],\tau_j)_{j=1}^N$. \emph{From here onward all partitions will be tagged.} The {\bf mesh width} of $P$ is defined as
\begin{displaymath}
  |P| \DEF \max_{1 \le j \le N}{(\sigma_j - \sigma_{j-1})}.
\end{displaymath}
Let $f : [a,b] \to V$ and $\eta : [a,b] \to W$ be given. For every tagged partition $P = ([\sigma_{j-1},\sigma_j],\tau_j)_{j=1}^N$ we introduce the {\bf Riemann-Stieltjes sums}
\begin{displaymath}
  S(f,P,d\eta) \DEF \sum_{j=1}^N{f(\tau_j)\cdot(\eta(\sigma_j) - \eta(\sigma_{j-1}))}
\end{displaymath}
and
\begin{displaymath}
  S(df,P,\eta) \DEF \sum_{j=1}^N{(f(\sigma_j) - f(\sigma_{j-1}))\cdot\eta(\tau_j)}.
\end{displaymath}
We note that if the bilinear product on $V \times W$ is symmetric (e.g. in the scalar case), then $S(df,P,\eta) = S(\eta,P,df)$ and there is no need to introduce a second Riemann-Stieltjes sum.
\begin{definition}
  \label{def:integration:1}
  If there exists $z \in Z$ with the property that for every $\EPS > 0$ there exists $\delta > 0$ such that $|P| \le \delta$ implies $\|S(f,P,d\eta) - z\| \le \EPS$ for any tagged partition $P$ of $[a,b]$, then {\bf $f$ is Riemann-Stieltjes integrable with respect to $\eta$} over $[a,b]$, written as $f \in S(\eta)$, and
\begin{displaymath}
  \int_a^b{f\,d\eta} \DEF z
\end{displaymath}
is the {\bf Riemann-Stieltjes integral of $f$ with respect to $\eta$}. Similarly, if there exists $z \in Z$ with the property that for every $\EPS > 0$ there exists $\delta > 0$ such that $|P| \le \delta$ implies $\|S(df,P,\eta) - z\| \le \EPS$ for any tagged partition $P$ of $[a,b]$, then {\bf $\eta$ is Riemann-Stieltjes integrable with respect to $f$} over $[a,b]$, written as $\eta \in S(f)$, and
\begin{displaymath}
  \int_a^b{df\,\eta} \DEF z
\end{displaymath}
is the {\bf Riemann-Stieltjes integral of $\eta$ with respect to $f$}. \hfill \QEDEX
\end{definition}
Often we will abbreviate the names \emph{Riemann} and \emph{Stieltjes} by the initials RS. It is trivial to check that the RS integral enjoys the usual bilinearity properties in $f$ and $\eta$. Also, if $c \in (a,b)$ and $\int_a^b{f\,d\eta}$ exists, then the integrals $\int_a^c{f\,d\eta}$ and $\int_c^b{f\,d\eta}$ exist as well and their sum equals the first integral. However, it is a shortcoming of the RS integral that the converse is generally \emph{not} true when $f$ and $\eta$ are allowed to be discontinuous. This phenomenon already occurs in the scalar case \cite[end of Appendix H]{Bartle2001}. 
\begin{remark}
  \label{rem:integration:1}
  In the special case that $Z = V$, $W = \KK$ and $v \cdot \alpha \DEF \alpha v$ for $v \in V$ and $\alpha \in \KK$, we recover the definition of the ordinary (but still $V$-valued) {\bf Riemann integral} upon setting $\eta(t) \DEF t$ for all $t \in [a,b]$. In that case, if $f \in S(\eta)$ then we write $\int_a^b{f\,d\eta} = \int_a^b{f(t)\,dt}$ as usual. \hfill \QEDEX
\end{remark}
The following lemma, implicitly used in the proof of \cite[Theorem 1.7 in Appendix I]{Diekmann1995} and stated as an exercise in \cite[Theorem H.2]{Bartle2001} for the case of scalar-valued, bounded $f$ and $\eta$, is important for the establishment of the existence of the RS integral.
\begin{lemma}[Cauchy criterion]
  \label{lem:integration:5}
  Let $f : [a,b] \to V$ and $\eta : [a,b] \to W$. Then $f \in S(\eta)$ if and only if for every $\EPS' > 0$ there exists $\delta' > 0$ such that if $P$ and $Q$ are tagged partitions of $[a,b]$ with $|P| \le \delta'$ and $|Q| \le \delta'$ then $\|S(f,P,d\eta) - S(f,Q,d\eta)\| \le \EPS'$. An analogous equivalence holds for RS integrability of $\eta$ with respect to $f$.
\end{lemma}
\begin{proof}
  We only prove the equivalence regarding RS integrability of $f$ with respect to $\eta$. The $\Leftarrow$ direction is trivial, so it remains to prove the reverse implication. Choose a sequence $(P_n)_n$ of tagged partitions such that $|P_n| \to 0$ as $n \to \infty$. Then the sequence $(S(f,P_n,d\eta))_n$ is Cauchy in the complete space $Z$. Denote its limit by $z$. Now let $\EPS > 0$ be given. By assumption there exists $\delta' > 0$ such that if $n \in \NN$ is chosen sufficiently large (as to guarantee that $|P_n| \le \delta'$) and $P$ is a tagged partition with $|P| \le \delta'$ then $\|S(f,P,d\eta) - S(f,P_n,d\eta)\| \le \frac{\EPS}{2}$. By choosing $n$ even larger if necessary, we also have $\|S(f,P_n,d\eta) - z\| \le \frac{\EPS}{2}$. Hence if $|P| \le \delta'$ then $\|S(f,P,d\eta) - z\| \le \EPS$ by the triangle inequality, as required.
\end{proof}
For a converse to the following theorem, we refer to \cite[\S 4]{Gowurin1936}.
\begin{theorem}
  \label{thm:integration:1}
  Let $\eta \in \BV([a,b],W)$. If $f \in S(\eta)$ then
  \begin{equation}
    \label{eq:integration:7}
    \Bigl\|\int_a^b{f\,d\eta}\Bigr\| \le \sup_{t \in [a,b]}\|f(t)\| V_a(\eta)(b).
  \end{equation}
  If $f \in C([a,b],V)$ then $f \in S(\eta)$.
\end{theorem}
\begin{proof}
  The inequality follows trivially by considering an approximating sequence $(P_n)_n$ of tagged partitions and using the definition of the total variation. The proof of RS integrability of $f \in C([a,b],V)$ with respect to $\eta$ uses \cref{lem:integration:5} and is a direct generalization of the proof for the scalar case, see \cite[Theorem 1.7 in Appendix I]{Diekmann1995} or \cite[Theorem H.3]{Bartle2001}. 
\end{proof}
The previous theorem implies that every $\eta \in \BV([a,b],\STAR{V})$ defines an element of the dual space of $C([a,b],V)$. Our interest in RS integrals stems mainly from the fact that the converse holds as well, which is the content of the upcoming \cref{thm:integration:2}.
\begin{definition}
  \label{def:integration:3}
  A function $\eta \in \BV([a,b],W)$ is said to be of {\bf normalized bounded variation} if $\eta(0) = 0$ and $\eta$ is continuous from the right on the open interval $(a,b)$. Let $\NBV([a,b],W)$ be the normed vector space of all functions of normalized bounded variation, endowed with the total variation norm $\|\eta\| \DEF V_a(\eta)(b)$. \hfill \QEDEX
\end{definition}
\begin{theorem}[{Riesz, Gowurin \cite[\S 6]{Gowurin1936}}]
  \label{thm:integration:2}
  For every $\STAR{f} \in \STAR{C([a,b],V)}$ there exists a unique $\eta \in \NBV([a,b],\STAR{V})$ such that 
  \begin{displaymath}
    \PAIR{f}{\STAR{f}} = \int_a^b{f\,d\eta}, \qquad \forall\,f \in C([a,b],V),
  \end{displaymath}
  and $\|\STAR{f}\| = \|\eta\|$. In particular, $\NBV([a,b],\STAR{V})$ is a Banach space.
\end{theorem}

We conclude this section with a few facts that are helpful for the manipulation of RS integrals. The proof of the first result follows readily from the definitions.
\begin{proposition}
  \label{prop:integration:5}
  Let $f : [a,b] \to \KK$, $v \in V$ and $\eta : [a,b] \to W$. Suppose $f \in S(\eta)$ for the bilinearity between $\KK$ and $W$. Then $f v \in S(\eta)$ for the bilinearity between $V$ and $W$ and 
  \begin{displaymath}
    \int_a^b{f v \,d\eta} = v\int_a^b{f\,d\eta}.
  \end{displaymath}
\end{proposition}
\begin{proposition}[Integration by parts]
  \label{prop:integration:1}
  Let $f : [a,b] \to V$ and $\eta : [a,b] \to W$. If $f \in S(\eta)$ then $\eta \in S(f)$ and
  \begin{equation}
    \label{eq:integration:2}
    \int_a^b{f\,d\eta} + \int_a^b{df\,\eta} = f(t)\cdot\eta(t)\Bigr|_{t=a}^{t=b}.
  \end{equation}
\end{proposition}
\begin{proof}
  The proof is adapted and slightly modified from \cite[Theorem H.5]{Bartle2001} which applied to the scalar case. It is a bit tedious, but since the proof of the scalar result in \cite[Theorem 1.8 in Appendix I]{Diekmann1995} is omitted (the suggested identity seems incorrect although it does contain a hint to the idea of the proof) and the result will be used quite often, we include a proof here.
  \begin{steps}
  \item
    Let $P = ([\sigma_{j-1},\sigma_j],\tau_j)_{j=1}^N$ be a tagged partition of $[a,b]$ with the additional property that
    \begin{equation}
      \label{eq:integration:3}
      \tau_j < \tau_{j+1}, \qquad \forall\,j = 1,\ldots,N-1.
    \end{equation}
    A direct calculation shows that
    \begin{align*}
      &S(df,P,\eta) - f(t)\eta(t)\Bigr|_{t=a}^{t=b} = \sum_{j=1}^N{(f(\sigma_j) - f(\sigma_{j-1}))\cdot\eta(\tau_j)} - f(t)\cdot\eta(t)\Bigr|_{t=a}^{t=b}\\
      &= -\Bigl[f(a)\cdot(\eta(\tau_1) - \eta(a)) + \sum_{j=1}^{N-1}{f(\sigma_j)\cdot(\eta(\tau_{j+1})- \eta(\tau_j))} + f(b)\cdot(\eta(b) - \eta(\tau_N))\Bigr].
    \end{align*}
    We observe that $\sigma_j \in [\tau_j,\tau_{j+1}]$ for $j = 1,\ldots,N-1$. We define a tagged partition $Q$ of $[a,b]$ such that the term in brackets equals the RS sum $S(f,Q,d\eta)$. Let $\tau_0 \DEF a$ and $\tau_{N+1} \DEF b$ and set
    \begin{displaymath}
      Q \DEF 
      \begin{cases}
        ([\tau_j,\tau_{j+1}],\sigma_j)_{j=0}^N &\text{if } a < \tau_1 \text{ and } \tau_N < b,\\
        ([\tau_j,\tau_{j+1}],\sigma_j)_{j=1}^{N-1} &\text{if } a = \tau_1 \text{ and } \tau_N = b,\\
        ([\tau_j,\tau_{j+1}],\sigma_j)_{j=1}^N & \text{if } a = \tau_1 \text{ and } \tau_N < b,\\
        ([\tau_j,\tau_{j+1}],\sigma_j)_{j=0}^{N-1} & \text{if } a < \tau_1 \text{ and } \tau_N = b. 
      \end{cases}
    \end{displaymath}
    Then indeed
    \begin{equation}
      \label{eq:integration:1}
      S(df,P,\eta) - f(t)\cdot\eta(t)\Bigr|_{t=a}^{t=b} = -S(f,Q,d\eta),
    \end{equation}
    and moreover $|Q| \le 2 |P|$.
  \item
    If $P' = ([\sigma_{j-1}',\sigma_j'],\tau_j')_{j=1}^M$ is a tagged partition of $[a,b]$ and $\tau_k' = \tau_{k+1}'$ for some $1 \le k \le M - 1$, then one easily verifies that
    \begin{displaymath}
      S(df,P',\eta) = \sum_{\substack{j=1\\j\neq k,k+1}}^M{(f(\sigma_j') - f(\sigma_{j-1}'))\cdot\eta(\tau_j')} + (f(\sigma_{k+1}') - f(\sigma_{k-1}'))\cdot\eta(\tau_k'),
    \end{displaymath}
    which implies the existence of a tagged partition $P$ of $[a,b]$ such that \cref{eq:integration:3} holds for its sample points, $S(df,P',\eta) = S(df,P,\eta)$ and $|P| \le 2|P'|$.
  \item
    Let $\EPS > 0$ be given and let $\delta > 0$ be provided by the definition of RS integrability of $f$ with respect to $\eta$. Let $P'$ be a tagged partition with $|P'| \le \frac{\delta}{4}$. Let $P$ be the tagged partition derived from $P'$ as in the previous step and let $Q$ be such that \cref{eq:integration:1} holds. Then $|Q| \le 2|P| \le 4|P'| \le \delta$ and by \cref{eq:integration:1} we have 
    \begin{displaymath}
      \Bigl\|S(df,P',\eta) - \bigl[f(t)\eta(t)\Bigr|_{t=a}^{t=b} - \int_a^b{f\,d\eta}\bigr] \Bigr\| = \Bigl\|  \int_a^b{f\,d\eta} - S(f,Q,d\eta) \Bigr\| \le \EPS.
    \end{displaymath}
    This proves that $\eta \in S(f)$ and \cref{eq:integration:2} holds. \hfill \qedhere
  \end{steps}
\end{proof}
For the following result we remark that when the integrand is vector valued, Riemann integrability does \emph{not} imply Bochner integrability \cite[Example 1.9.7]{Arendt2011} since Riemann integrable functions need not be strongly measurable. This is in contrast with the scalar case.
\begin{proposition}
  \label{prop:integration:6}
  Every function $\eta : [a,b] \to W$ of bounded variation is Riemann integrable as well as Bochner integrable.
\end{proposition}
\begin{proof}
  For Riemann integrability we apply \cref{prop:integration:1} with $f(t) = t$. This gives $\eta \in S(f)$ so $\eta$ is Riemann integrable over $[a,b]$.
  
  For Bochner integrability we use Pettis' measurability theorem \cite[Theorem 1.1.1]{Arendt2011} to verify that $\eta$ is strongly measurable, also see \cite[Lemma 2.1]{Greiner1992} for what inspired the current proof. Once this is done, then \cref{cor:integration:1} and Bochner's theorem \cite[Theorem 1.1.4]{Arendt2011} together imply that $\eta$ is Bochner integrable.

  First we verify that $\eta$ is almost separably valued. Let $D$ be the set of points in $[a,b]$ where $\eta$ is not continuous. Then $D$ is countable by \cref{cor:integration:2}, so $D$ has Lebesgue measure zero. Also, $[a,b] \setminus D$ is separable, so its image under the continuous restriction of $\eta$ is separable in $W$. This means that $\eta$ is almost separably valued.

  It remains to check that $\eta$ is weakly measurable. It follows directly from the definition of the variation function \cref{eq:integration:V} that for every $\STAR{w} \in \STAR{W}$ the function $\STAR{w} \circ \eta : [a,b] \to \KK$ is of bounded variation. Splitting $\eta$ into real and imaginary parts (in case $\KK = \CC$) and then writing each part as the difference of two nondecreasing functions on $[a,b]$, we see that $\STAR{w} \circ \eta$ is Lebesgue measurable on $[a,b]$, so $\eta$ itself is indeed weakly measurable. Pettis' measurability theorem then implies that $\eta$ is strongly measurable.
\end{proof}
\begin{proposition}
  \label{prop:integration:3}
  Let $f : [a,b] \to V$ and $g: [a,b] \to W$ be Riemann integrable over $[a,b]$. Then $f \cdot g : [a,b] \to Z$ is Riemann integrable over $[a,b]$.
\end{proposition}
%
Next we formulate a counterpart to \cite[Theorem 1.9 in Appendix I]{Diekmann1995} which allows us to evaluate RS integrals as Riemann integrals. For similar results in case either $V$ or $W$ is the scalar field, see \cite[\S 1.9]{Arendt2011}. We observe that by \cref{prop:integration:6} the following proposition applies in particular when $\eta \in \BV([a,b],W)$.
\begin{proposition}
  \label{prop:integration:2}
  Let $f \in C^1([a,b],V)$ and let $\eta : [a,b] \to W$ be Riemann integrable. Then $\eta \in S(f)$ and
  \begin{equation}
    \label{eq:integration:4}
    \int_a^b{df\,\eta} = \int_a^b{f'(t)\cdot \eta(t)\,dt},
  \end{equation}
  where the right-hand side is a Riemann integral.
\end{proposition}
\begin{proof}
  \Cref{prop:integration:3} implies that $f' \cdot \eta$ is Riemann integrable. Let $P = ([\sigma_{j-1},\sigma_j],\tau_j)_{j=1}^N$ be a tagged partition of $[a,b]$. Then
\begin{align*}
  \Bigl\|S(df,P,\eta) - \int_a^b{f'(t)\cdot\eta(t)\,dt}\Bigr\| &\le \Bigl\|S(df,P,\eta) - \sum_{j=1}^N{f'(\tau_j)\cdot \eta(\tau_j) (\sigma_j - \sigma_{j-1})}\Bigr\|\\
  &+ \Bigl\|\sum_{j=1}^N{f'(\tau_j)\cdot \eta(\tau_j) (\sigma_j - \sigma_{j-1})} - \int_a^b{f'(t)\cdot\eta(t)\,dt} \Bigr\|,
\end{align*}
and we observe that the second term in the right-hand side tends to zero as $|P| \to 0$. Let $\EPS > 0$ be arbitrary. By the uniform continuity of $f'$ on $[a,b]$ there exists $\delta > 0$ such that $\|f'(s) - f'(t)\| \le \EPS$ whenever $s,t \in [a,b]$ and $|s - t| \le \delta$. Then $|P| \le \delta$ implies that the first term in the right-hand side equals
\begin{align*}
  \Bigl\| \sum_{j=1}^N(f(\sigma_i) - f(\sigma_{i-1}))\cdot\eta(\tau_i) &- \sum_{j=1}^N{f'(\tau_j)\cdot \eta(\tau_j) (\sigma_j - \sigma_{j-1})}\Bigr\|\\
  &= \Bigl\|\sum_{j=1}^N{\int_{\sigma_{j-1}}^{\sigma_j}{(f'(s) - f'(\tau_j))\cdot\eta(\tau_j)\,ds}  } \Bigr\|\\
  &\le \EPS \|\eta\|_{\infty}(b - a).
\end{align*}
Now take the limit $\EPS \downarrow 0$.
\end{proof}
We also have the following variation. In contrast with \cref{prop:integration:2} now the integrator is only assumed to have a Bochner integrable derivative \cite[Proposition 1.2.2]{Arendt2011} but the {\it integrand} is supposed to be continuous.
\begin{proposition}
  \label{prop:integration:4}
  Let $f \in C([a,b],V)$ and $h \in \LP{1}([a,b],W)$ and define $\eta : [a,b] \to W$ by
  \begin{displaymath}
    \eta(t) \DEF \int_a^t{h(s)\,ds}, \qquad \forall\,t \in [a,b].
  \end{displaymath}
  Then $f \in S(\eta)$ and
  \begin{equation}
    \label{eq:integration:10}
    \int_a^b{f\,d\eta} = \int_a^b{f(s)\cdot h(s)\,ds},
  \end{equation}
  where the right-hand side is a Bochner integral.
\end{proposition}
\begin{proof}
  First we observe that $\eta \in \BV([a,b],W)$ by \cref{lem:duality:1}. Since $f$ is continuous, \cref{thm:integration:1} then implies that $f \in S(\eta)$. Also, since $f$ is bounded and measurable, it follows that $f\cdot h$ is Bochner integrable over $[a,b]$. 
  \par
  Given $\EPS > 0$ by the uniform continuity of $f$ on $[a,b]$ there exists $\delta > 0$ such that $\|f(s) - f(t)\| \le \EPS$ for all $s,t \in [a,b]$ with $|s - t| \le \delta$. Let $P = ([\sigma_{j-1},\sigma_j],\tau_j)_{j=1}^N$ be a tagged partition of $[a,b]$ with $|P| \le \delta$. Then
  \begin{align*}
    \Bigl\|S(f,P,d\eta) - \int_a^b{f(s)\cdot h(s)\,ds}\Bigr\| &= \Bigl\|\sum_{j=1}^N{f(\tau_j)\cdot(\eta(\sigma_j) - \eta(\sigma_{j-1}))} - \int_a^b{f(s)\cdot h(s)\,ds}\Bigr\|\\
    &= \Bigl\|\sum_{j=1}^N{\int_{\sigma_{j-1}}^{\sigma_j}{f(\tau_j)\cdot h(s)\,ds}} - \sum_{j=1}^N{\int_{\sigma_{j-1}}^{\sigma_j}{f(s)\cdot h(s)\,ds}}\Bigr\|\\
    &\le \sum_{j=1}^N{\int_{\sigma_{j-1}}^{\sigma_j}{\|f(\tau_j) - f(s)\| \|h(s)\|\,ds}}\\
    &\le \EPS \int_a^b{\|h(s)\|\,ds},
  \end{align*}
  and the result follows by letting $\EPS \downarrow 0$.
\end{proof}

In the main text we often encounter the special case $W = \STAR{V}$, $Z = \KK$ and $v \cdot \STAR{v} \DEF \PAIR{v}{\STAR{v}}$. We remark that in this setting the right-hand side of \cref{eq:integration:10} reduces to an ordinary (i.e. scalar valued) Lebesgue integral. 
